\documentclass[12pt,a4paper]{amsart}
\usepackage{enumerate,amsmath,amsthm,amssymb,latexsym}
\usepackage[all]{xy}
\usepackage[parfill]{parskip}

\usepackage{amsmath,amssymb,xspace,amsthm}
\newtheorem{theorem}{Theorem}
\newtheorem{remark}[theorem]{Remark}
\newtheorem{lemma}[theorem]{Lemma}
\newtheorem{proposition}[theorem]{Proposition}
\newtheorem{corollary}[theorem]{Corollary}
\numberwithin{equation}{section}

\newcommand{\tto}{\twoheadrightarrow}

\begin{document}
\title[Category $\mathcal{O}$ for the Schr{\"o}dinger algebra]{Category $\mathcal{O}$ for the Schr{\"o}dinger algebra}
\author[B.~Dubsky, R.~L{\"u}, V.~Mazorchuk and K.~Zhao]{Brendan Dubsky, 
Rencai L{\"u}, Volodymyr Mazorchuk and Kaiming Zhao}
\date{\today
{\bf 2010  Math. Subj. Class.:} 17B10, 17B65, 17B66, 17B68}

\begin{abstract}
We study category $\mathcal{O}$ for the (centrally extended) Schr{\"o}dinger algebra.
We determine the quivers for all blocks and relations for blocks of nonzero central charge.
We also describe the quiver and relations for the finite dimensional part of $\mathcal{O}$.
We use this to determine the center of the universal enveloping algebra and annihilators
of Verma modules. Finally, we classify primitive ideals of the universal enveloping algebra 
which intersect the center of the centrally extended Schr{\"o}dinger algebra trivially.
\end{abstract}
\maketitle

\section{Introduction and description of the results}\label{s1}

The Schr{\"o}dinger Lie group describes symmetries of the free particle Schr{\"o}dinger equation, see \cite{Pe}. 
The corresponding Lie algebra is called the Schr{\"o}dinger algebra,  see \cite{DDM}.
In the $1+1$-dimensional space-time this algebra is, roughly, a semi-direct product of the simple Lie 
algebra $\mathfrak{sl}_2$ with its simple $2$-dimensional representation (the latter forms an abelian ideal). 
This Lie algebra admits a universal $1$-dimensional central extension which is called the centrally 
extended Schr{\"o}dinger algebra or, simply, the Schr{\"o}dinger algebra, abusing the language.

Some basics of the representation theory of the Schr{\"o}dinger algebra were studied in \cite{DDM,DDM2},
including description of simple highest weight modules. Recently there appeared a number of papers 
studying various aspects of the representation theory of the Schr{\"o}dinger algebra, see
\cite{AD,LMZ1,LMZ,Du,Wu,WZ1,WZ2}. In particular, \cite{Du} classifies all simple modules
over the Schr{\"o}dinger algebra which are weight and have finite dimensional weight spaces.

The present paper started with the observation that the claim of 
\cite[Theorem~1.1(1)]{WZ1} contradicts \cite[page 244]{Pe}
and a natural subsequent attempt to repair the main result of \cite{WZ1} which claims to describe annihilators
of Verma modules over the Schr{\"o}dinger algebra. In the classical situation of simple Lie algebras,
study of annihilators of Verma modules usually follows the study of the BGG category $\mathcal{O}$ and
its equivalent realization using Harish-Chandra bimodules. This naturally led us to the problem of understanding
category $\mathcal{O}$ for the Schr{\"o}dinger algebra. This is the main objective of the present paper.

Making a superficial parallel with the theory of affine Lie algebras, it turns out that the representation theory of 
the Schr{\"o}dinger algebra splits into two very different cases, namely the case of nonzero central charge 
and the one of the zero central charge, where by the {\em central charge} we, as usual, mean the
eigenvalue of the (unique up to scalar) central element of the Schr{\"o}dinger algebra
(note that such an eigenvalue is unique for all simple modules). For nonzero central charge our results
are complete, whereas for zero central charge we get less information, however, involving much more
complicated arguments. Nevertheless, we derive enough properties of $\mathcal{O}$ to be able to describe 
the center of the universal enveloping algebra of the Schr{\"o}dinger algebra and
annihilators of Verma modules, repairing the main results of \cite{WZ1}. Along the way we also
describe the ``finite dimensional'' part of $\mathcal{O}$ which, in contrast with the classical case, is
no longer a semi-simple category. Our description, in particular,  implies that the category of finite dimensional 
modules over the Schr{\"o}dinger algebra has wild representation type (cf. \cite{Mak}).

The paper is organized as follows: in Section~\ref{s2} we collected all necessary preliminaries.
Section~\ref{s3} studies basics on category $\mathcal{O}$ and describes blocks of nonzero
central charge. Section~\ref{s4} studies blocks of zero central charge and the ``finite dimensional''
part of $\mathcal{O}$. As a technical tool we also introduce a natural graded version of 
$\mathcal{O}$ (which makes sense only for zero central charge). Section~\ref{s5} contains several applications,
in particular,  description of the center of the universal enveloping algebra of the 
Schr{\"o}dinger algebra and description of annihilators of Verma modules.
In Section~\ref{s6} we outline the setup to study Harish-Chandra bimodules for 
the Schr{\"o}dinger algebra and apply it to obtain a classification of primitive ideals with nonzero central charge.
\vspace{2mm}

After the paper was finished we were informed that the fact that the results of \cite{WZ1} are not correct
was recently pointed out in \cite{WZ3}.
\vspace{5mm}

\noindent{\bf Acknowledgements.}
R.L. is partially supported by NSF of China  (Grant 11371134) and 
Jiangsu Government Scholarship for Overseas Studies (JS-2013-313).\\
V. M. is partially supported by the Swedish Research Council.\\
K.Z. is partially supported by  NSF of China (Grant 11271109) and NSERC.\\
We thank Peter {\v S}emrl for informing us about \cite{WZ3}.
We also thank the referee for very helpful comments.

\section{The Schr{\"o}dinger algebra}\label{s2}

\subsection{Notation}\label{s2.0}

We denote by $\mathbb{N}$, $\mathbb{Z}_+$ and $\mathbb{C}$ the sets of positive integers, non-negative integers
and complex numbers, respectively. For a Lie algebra $\mathfrak{a}$ we denote by $U(\mathfrak{a})$ the
universal enveloping algebra of $\mathfrak{a}$. We also denote by $Z(\mathfrak{a})$ the center of $U(\mathfrak{a})$.
We denote by ${\underline{\hspace{2mm}}}^*$ the usual duality $\mathrm{Hom}_{\mathbb{C}}({}_-,\mathbb{C})$.
For an associative algebra $A$ we denote by $A\text{-}\mathrm{Mod}$ the category of all $A$-modules and
by $A\text{-}\mathrm{mod}$ the full subcategory of $A\text{-}\mathrm{Mod}$ consisting of all
finitely generated modules. For a Lie algebra $\mathfrak{a}$ we set 
\begin{displaymath}
\mathfrak{a}\text{-}\mathrm{Mod}:=U(\mathfrak{a})\text{-}\mathrm{Mod}\quad \text{ and }\quad
\mathfrak{a}\text{-}\mathrm{mod}:=U(\mathfrak{a})\text{-}\mathrm{mod}.
\end{displaymath}
We write $\otimes$ for $\otimes_{\mathbb{C}}$.

\subsection{Definition}\label{s2.1}

The {\em Schr{\"o}dinger algebra} $\mathfrak{s}$ is the complex Lie algebra with a basis 
$\{e,h,f,p,q,z\}$ where $z$ is central and the rest of the Lie bracket is given as follows:
\begin{equation}
\label{commrelations}
\begin{array}{lll}
\left[h,e\right]=2e, & \left[e,f\right]=h,  & \left[h,f\right]=-2f,\\
\left[e,q\right]=p, &  \left[e,p\right]=0, & \left[h,p\right]=p,\\
\left[f, p\right]=q, &\left[f,q\right]=0, &\left[h,q\right]=-q,\\
&\left[p,q\right]=z.
\end{array}
\end{equation}
The algebra $\mathfrak{s}$ is not semi-simple, its radical being the nilpotent ideal 
$\mathfrak{i}$ spanned by $p,q$ and $z$. Note that $\mathfrak{i}$ is a Heisenberg Lie algebra
while the quotient $\mathfrak{s}/\mathfrak{i}$ is isomorphic to the simple 
complex Lie algebra $\mathfrak{sl}_2$. The center of $\mathfrak{s}$ is spanned by $z$.
We denote by $\overline{\mathfrak{s}}$ the {\em centerless} Schr{\"o}dinger algebra
$\mathfrak{s}/\mathbb{C}z$.

To simplify notation we set $U=U(\mathfrak{s})$.
With respect to the adjoint action of $h$ we have the decomposition 
\begin{displaymath}
U=\bigoplus_{i\in\mathbb{Z}}U_i,\qquad \text{ where}\qquad  U_i:=\{u\in U\,\vert\, [h,u]=iu\}.
\end{displaymath}
Note that $U_iU_j\subset U_{i+j}$ for all $i,j\in\mathbb{Z}$.
The algebra $U$ is a noetherian domain (both, left and right).

\subsection{Casimir element}\label{s2.2}

Consider the classical Casimir element $\underline{\mathtt{c}}:=(h+1)^2+4fe$ in $U(\mathfrak{sl}_2)$ and
the following element in $U(\mathfrak{s})$:
\begin{multline*}
\mathtt{c}:=\underline{\mathtt{c}}z-2(fp^2-q^2e-2qp-hqp)+hz+z=\\=
(h^2+h+4fe)z-2(fp^2-eq^2-hpq).
\end{multline*}

The following statement verifies \cite[Formula~(3)]{AD} and \cite[Page~244]{Pe}.

\begin{lemma}\label{lem1}
We have $\mathtt{c}\in Z(\mathfrak{s})$.
\end{lemma}

\begin{proof}
Clearly, every summand of $\mathtt{c}$ is in $U_0$ and hence $\mathtt{c}\in U_0$. 
Further, using the facts that $(h+1)^2+4fe$ is a Casimir element
for $\mathfrak{sl}_2$ and $z$ is central in $\mathfrak{s}$, we have
\begin{displaymath}
\begin{array}{rcl}
[e,\mathtt{c}]&=&[e,-hz]-2[e,fp^2-eq^2-hpq]\\
&=&2ez-2([e,fp^2]-[e,eq^2]-[e,hpq])\\
&=&2ez-2(hp^2-epq-eqp+2epq-hp^2)\\
&=&0.
\end{array}
\end{displaymath}
Similarly one checks that $[f,\mathtt{c}]=0$.

Further, we have 
 \begin{displaymath}
\begin{array}{rcl}
[p,\mathtt{c}]&=&[p,h^2+h+4fe]z-2[p,fp^2-eq^2-hpq]\\
&=&(-ph-hp-p-4qe)z-2(-qp^2-2ezq+p^2q-hpz)\\
&=&0.
\end{array}
\end{displaymath}
Similarly one checks that $[q,\mathtt{c}]=0$. This shows that 
$\mathtt{c}\in Z(\mathfrak{s})$.
\end{proof}

\subsection{Cartan subalgebra}\label{s2.3}

Denote by $\mathfrak{h}$ the Cartan subalgebra of $\mathfrak{s}$, spanned by $h$ and $z$. The algebra 
$\mathfrak{h}$ is commutative and its adjoint action on $\mathfrak{s}$ is diagonalizable. 
Fix the basis  $\{h^{\checkmark},z^{\checkmark}\}$ in $\mathfrak{h}^*$ which is dual to the basis $\{h,z\}$.
For $\alpha\in \mathfrak{h}^*$ set 
\begin{displaymath}
\mathfrak{s}_\alpha:=\{x\in\mathfrak{s}\,\vert\, [H,x]=\alpha(H)x 
\text{ for all }H\in \mathfrak{h}\}.
\end{displaymath}
Then we have
\begin{displaymath}
\mathfrak{s}=\mathfrak{s}_{-2h^{\checkmark}}\oplus
\mathfrak{s}_{-h^{\checkmark}}\oplus\mathfrak{s}_0
\oplus\mathfrak{s}_{h^{\checkmark}}\oplus\mathfrak{s}_{2h^{\checkmark}}  
\end{displaymath}
where $\mathfrak{s}_0=\mathfrak{h}$ has dimension two while all other spaces are one-\-di\-men\-si\-o\-nal.
We set $R:=\{\pm 2h^{\checkmark},\pm h^{\checkmark}\}$ and call the elements of $R$ {\em roots}
of $\mathfrak{s}$. Note that $R$ is a root system (not reduced) in its linear span.

As usual, we denote by $\rho$ the half of the sum of all positive roots, that is 
$\rho=\frac{3}{2}h^{\checkmark}$. Let $W$ be the Weyl group of $R$, that is the group consisting of the
identity and the linear transformation $r$ defined as follows:
\begin{displaymath}
r(z^{\checkmark})=z^{\checkmark}\quad\text{ and }\quad r(h^{\checkmark})=-h^{\checkmark}.
\end{displaymath}
Then $W$ naturally acts on $\mathfrak{h}^*$ and we also have the $\rho$-shifted {\em dot-action} given by
$w\cdot \lambda=w(\lambda+\rho)-\rho$ for $w\in W$ and $\lambda\in \mathfrak{h}^*$.

\subsection{Triangular decomposition}\label{s2.4}

Write  
\begin{displaymath}
R=R_-\cup R_+,\quad\text{ where } \quad  R_{+}:=\{2h^{\checkmark},h^{\checkmark}\}\quad\text{ and } \quad R_-=-R_+ 
\end{displaymath}
and set
\begin{displaymath}
\mathfrak{n}_{\pm}:=\bigoplus_{\alpha\in R_{\pm}}\mathfrak{s}_\alpha.
\end{displaymath}
Then the decomposition
\begin{equation}\label{eq1}
\mathfrak{s}=\mathfrak{n}_{-}\oplus\mathfrak{h}\oplus \mathfrak{n}_{+} 
\end{equation}
is a {\em triangular decomposition} of $\mathfrak{s}$ in the sense of \cite{MP}. 
This decomposition implies the following decomposition
of $U$ as $U(\mathfrak{n}_{-})\text{--}U(\mathfrak{n}_{+})$-bimodules:
\begin{displaymath}
U\cong U(\mathfrak{n}_{-})\otimes U(\mathfrak{h})\otimes U(\mathfrak{n}_{+}).
\end{displaymath}
We also set $\mathfrak{b}:=\mathfrak{h}\oplus \mathfrak{n}_{+}$.

\subsection{Weight modules}\label{s2.5}

As usual, an $\mathfrak{s}$-module $M$ is called a {\em weight} module provided that 
\begin{displaymath}
M\cong\bigoplus_{\lambda\in \mathfrak{h}^*} M_{\lambda},\,\,\,\text{ where }\,\,\,
M_{\lambda}:=\{v\in M\,\vert\, H\cdot v=\lambda(H)v\text{ for all }H\in\mathfrak{h}\}.
\end{displaymath}
Elements $\lambda\in \mathfrak{h}^*$ are called {\em weights} and for $\lambda\in \mathfrak{h}^*$
the space $M_{\lambda}$ is the corresponding {\em weight space}. We denote by $\mathrm{supp}(M)$
the {\em support} of $M$, that is the set of all $\lambda\in \mathfrak{h}^*$ such that $M_{\lambda}\neq 0$.

Since the adjoint action of $\mathfrak{h}$ on $\mathfrak{s}$ is diagonalizable, it follows
that a module generated by a weight vector is a weight module. We denote by $\mathfrak{W}$ the full
subcategory of $U\text{-}\mathrm{Mod}$ consisting of all weight modules.

It is very natural to introduce another class of ``weight'' modules.
An $\mathfrak{s}$-module $M$ is called an {\em $h$-weight} module provided that 
\begin{displaymath}
M\cong\bigoplus_{{\dot h}\in \mathbb{C}} M_{{\dot h}},\,\,\,\text{ where }\,\,\,
M_{\lambda}:=\{v\in M\,\vert\, h\cdot v={\dot h}v\}.
\end{displaymath}
Elements ${\dot h}\in \mathbb{C}$ are called {\em $h$-weights} and for ${\dot h}\in \mathbb{C}$
the space $M_{{\dot h}}$ is the corresponding {\em $h$-weight space}. We denote by $\mathrm{supp}_h(M)$
the {\em support} of $M$, that is the set of all ${\dot h}\in \mathbb{C}$ such that $M_{{\dot h}}\neq 0$.
Again, a module generated by an $h$-weight vector is an $h$-weight module. We denote by $\mathfrak{V}$ the full
subcategory of $U\text{-}\mathrm{Mod}$ consisting of all $h$-weight modules. Clearly,
$\mathfrak{W}$ is a full subcategory of $\mathfrak{V}$.

As $U$ is a finitely generated algebra over an uncountable algebraically closed field 
$\mathbb{C}$, every central element acts as a scalar on each
simple $U$-module by Schur's lemma (cf. \cite[Theorem~4.7]{Ma}). It follows that every simple
$h$-weight module is a weight module. In particular, simple objects in $\mathfrak{V}$ and $\mathfrak{W}$
coincide.

\section{Category $\mathcal{O}$}\label{s3}

\subsection{Definition}\label{s3.1}

As usual (see \cite{BGG,MP,Hu}) we define the category $\mathcal{O}$ associated to the 
triangular decomposition \eqref{eq1} as the full subcategory of $U\text{-}\mathrm{mod}\cap\mathfrak{W}$
consisting of all modules $M$ on which the action of $U(\mathfrak{n}_+)$ is {\em locally finite} in the sense
that $\dim U(\mathfrak{n}_+)v<\infty$ for all $v\in M$.

Directly from the definition it follows that category $\mathcal{O}$ is closed under taking quotients and 
finite direct sums. As $U$ is noetherian, category $\mathcal{O}$ is also closed under taking submodules.
It follows that category $\mathcal{O}$ is abelian. Furthermore, for $M\in \mathcal{O}$ there is a finite set 
$\{\lambda_1,\dots,\lambda_k\}\subset\mathfrak{h}^*$ such that 
\begin{displaymath}
\mathrm{supp}(M)\subset \bigcup_{i=1}^k \big(\lambda_i-\mathbb{Z}_+ R_+\big).
\end{displaymath}
As $M$ is finitely generated and $\mathfrak{h}$-weight spaces of the adjoint action of $\mathfrak{h}$ 
on $U(\mathfrak{n}_+)$ are finite dimensional, 
it follows that $\dim M_{\lambda}<\infty$ for all $\lambda\in\mathfrak{h}^*$
and therefore $\dim \mathrm{Hom}_{\mathcal{O}}(M,N)<\infty$ for all $M,N\in\mathcal{O}$.
Consequently, $\mathcal{O}$ is idempotent split and hence Krull-Schmidt.

\subsection{Verma modules}\label{s3.2}

For $\lambda\in\mathfrak{h}^*$ denote by $\mathbb{C}_{\lambda}$ the one-dimensional $\mathfrak{b}$-module
with generator $v_{\lambda}$ and the action given by 
\begin{displaymath}
\mathfrak{n}_+ \mathbb{C}_{\lambda}=0,\qquad H\cdot v_{\lambda}=\lambda(H) v_{\lambda}\text{ for all } H\in\mathfrak{h}.
\end{displaymath}
The {\em Verma module} is defined, as usual, as follows (see \cite{Di,Hu} for the classical case
and \cite{DDM} for the case of the algebra $\mathfrak{s}$):
\begin{displaymath}
\Delta(\lambda):=\mathrm{Ind}_{\mathfrak{b}}^{\mathfrak{s}}\mathbb{C}_{\lambda}\cong
U\bigotimes_{U(\mathfrak{b})}\mathbb{C}_{\lambda}.
\end{displaymath}
By abuse of notation we denote the canonical generator $1\otimes v_{\lambda}$ of $\Delta(\lambda)$
simply by $v_{\lambda}$. It follows directly from the definition that $\Delta(\lambda)$ is a weight
module with support 
\begin{displaymath}
\mathrm{supp}(\Delta(\lambda))=\lambda-\mathbb{Z}_+R_+=
\{\lambda-ih^{\checkmark}\,\vert\, i\in \mathbb{Z}_+\}
\end{displaymath}
and, moreover, $\dim \Delta(\lambda)_{\lambda-ih^{\checkmark}}=\lfloor\frac{i}{2}\rfloor+1$ for all $i\in \mathbb{Z}_+$.
The weight $\lambda$ is called the {\em highest weight} of $\Delta(\lambda)$.

As usual (cf. \cite[Proposition~7.1.8(iv)]{Di}), we have $\mathrm{End}_{\mathcal{O}}(\Delta(\lambda))\cong \mathbb{C}$,
in particular, $\Delta(\lambda)$ is indecomposable. Moreover, $\Delta(\lambda)$ has a unique maximal
submodule $K(\lambda)$ (which is the sum of all submodules $N$ of $\Delta(\lambda)$ satisfying the
condition $N_{\lambda}=0$) and hence the unique simple quotient $L(\lambda)=\Delta(\lambda)/K(\lambda)$.
The module $L(\lambda)$ is the {\em simple highest weight module} with highest weight $\lambda$. As usual, see 
\cite{MP,Hu}, each $L(\lambda)$ is a simple object of $\mathcal{O}$ and each simple object of $\mathcal{O}$
is isomorphic to $L(\lambda)$ for a unique $\lambda\in\mathfrak{h}^*$. 
 
For $\lambda\in  \mathfrak{h}^*$ we denote by $\vartheta_{\lambda}\in\mathbb{C}$ the scalar corresponding to
the action of the central element $\mathtt{c}$ on $\Delta(\lambda)$.

As a $U(\mathfrak{n}_-)$-module, each Verma module is free of rank $1$. Since $U(\mathfrak{n}_-)$ is 
a domain, it follows that each nonzero homomorphism between Verma modules is injective. Moreover, 
each Verma module has Gelfand-Kirillov dimension $\dim \mathfrak{n}_-=2$.

\subsection{Rough block decomposition}\label{s3.3}

For $\xi\in \mathfrak{h}^*/\mathbb{Z}R$ denote by $\mathcal{O}[\xi]$ the full subcategory of 
$\mathcal{O}$ consisting of all $M$ such that $\mathrm{supp}(M)\subset \xi$. Then we have
the following decomposition:
\begin{displaymath}
\mathcal{O}\cong\bigoplus_{\xi\in \mathfrak{h}^*/\mathbb{Z}R}\mathcal{O}[\xi]. 
\end{displaymath}
Given $\xi\in \mathfrak{h}^*/\mathbb{Z}R$, the value ${\dot z}={\dot z}_{\xi}:=\lambda(z)$,
$\lambda\in\xi$, does not depend on the choice of $\lambda$. It is called the {\em central charge}
of $\mathcal{O}[\xi]$ (and of any object in $\mathcal{O}[\xi]$).

\subsection{Blocks of nonzero central charge and not half-integral weights}\label{s3.4}

\begin{lemma}\label{lem2}
Let $\xi\in \mathfrak{h}^*/\mathbb{Z}R$ be of nonzero central charge.
Let $n\in\mathbb{Z}$ and ${\dot h}=\lambda(h)$ for some $\lambda\in\xi$.
Then $\vartheta_{\lambda}=\vartheta_{\lambda-nh^{\checkmark}}$ if and only if $\dot{h}=\frac{n-3}{2}$.
\end{lemma}

\begin{proof}
From the definition of $\mathtt{c}$, for any $\mu\in \mathfrak{h}^*$ we have
\begin{displaymath}
 \mathtt{c}\cdot v_{\mu}=(\mu(h)^2+3\mu(h)+2)\mu(z) v_{\mu}
\end{displaymath}
and the claim follows by comparing
the corresponding expressions for $\lambda$ and $\lambda-nh^{\checkmark}$.
\end{proof}

Let $\xi\in \mathfrak{h}^*/\mathbb{Z}R$ be of nonzero central charge. If $\lambda(h)\not\in\frac{1}{2}\mathbb{Z}$
for any $\lambda\in \xi$, then for any $\lambda\in \xi$ let $\mathcal{O}[\xi]_{\lambda}$ denote the Serre subcategory 
of  $\mathcal{O}[\xi]$ generated by $\Delta(\lambda)$.

\begin{proposition}\label{lem3}
Let $\xi\in \mathfrak{h}^*/\mathbb{Z}R$ be of nonzero central charge. Assume that 
$\lambda(h)\not\in\frac{1}{2}\mathbb{Z}$ for any $\lambda\in \xi$. Then we have the following:
\begin{enumerate}[$($i$)$]
\item\label{lem3.1} The module $\Delta(\lambda)$ is simple for any $\lambda\in \xi$. 
\item\label{lem3.2} We have the decomposition 
\begin{displaymath}
\mathcal{O}[\xi]\cong\bigoplus_{\lambda\in \xi}\mathcal{O}[\xi]_{\lambda} 
\end{displaymath} 
\item\label{lem3.3} We have  $\mathcal{O}[\xi]_{\lambda} \cong\mathbb{C}\text{-}\mathrm{mod}$ 
for any $\lambda\in \xi$, more precisely, the functor defined on objects as $N\mapsto N_{\lambda}$
and on morphisms in the obvious way provides an equivalence between $\mathcal{O}[\xi]_{\lambda}$ and the
category of finite dimensional complex vector spaces.
\end{enumerate}
\end{proposition}

\begin{proof}
Let $N$ be a proper submodule of $\Delta(\lambda)$. Then it has a nonzero highest weight vector of highest
weight $\lambda-ih^{\checkmark}$ for some $i\in\mathbb{N}$. But then 
$\vartheta_{\lambda}=\vartheta_{\lambda-ih^{\checkmark}}$ and we get a contradiction with Lemma~\ref{lem2}.
This proves claim~\eqref{lem3.1}. As $\mathtt{c}$ is central and has different eigenvalues on
$\Delta(\lambda)$ and $\Delta(\mu)$ for different $\lambda,\mu\in \xi$, we get claim~\eqref{lem3.2}.

The weight $\lambda$ is the highest weight for any $N\in \mathcal{O}[\xi]_{\lambda}$. By adjunction, we have
\begin{displaymath}
\mathrm{Hom}_{\mathcal{O}}(\Delta(\lambda),N)\cong N_{\lambda} 
\end{displaymath}
which means that the functor 
\begin{displaymath}
\mathrm{Hom}_{\mathcal{O}}(\Delta(\lambda),{}_-):\mathcal{O}[\xi]_{\lambda}\to \mathbb{C}\text{-}\mathrm{mod} 
\end{displaymath}
is isomorphic to the exact functor $N\mapsto N_{\lambda}$. Therefore the (unique up to isomorphism)
simple object $\Delta(\lambda)\in \mathcal{O}[\xi]_{\lambda}$ is projective. This implies
claim~\eqref{lem3.3} and completes the proof.
\end{proof}

\subsection{Projective functors}\label{s3.5}

For each finite dimensional $\mathfrak{sl}_2$-module $V$, 
viewed as an $\mathfrak{s}$-module via the canonical projection
$\mathfrak{s}\tto \mathfrak{s}/\mathfrak{i}\cong \mathfrak{sl}_2$, we have the functor 
\begin{displaymath}
\mathrm{F}_V:=V\otimes{}_-:\mathcal{O}\to \mathcal{O} 
\end{displaymath}
which preserves $\mathcal{O}[\xi]$ for every $\xi\in \mathfrak{h}^*/\mathbb{Z}R$.
As usual (see \cite[2.1(d)]{BG} or \cite[Lemma~3.71]{Ma}), the functor $\mathrm{F}_V$ is both left and right 
adjoint to itself. In particular, it sends projective objects to projective objects and injective objects
to injective objects.

\subsection{Duality}\label{s3.8}

Let $\sigma$ be the unique involutive anti-automorphism of $\mathfrak{s}$ satisfying $\sigma(e)=-f$, 
$\sigma(p)=q$ and $\sigma(z)=z$. For $M\in\mathcal{O}$ the space 
\begin{displaymath}
M^{\star}:=\bigoplus_{\lambda\in\mathfrak{h}^*}M_{\lambda}^{*} 
\end{displaymath}
becomes an $\mathfrak{s}$-module via $(x\cdot g)(v):=g(\sigma(x)v)$, where $x\in \mathfrak{s}$,
$g\in M^{\star}$ and $v\in M$. This defines an exact, contravariant and involutive
functor ${{}_-}^{\star}:\mathcal{O}\to \mathcal{O}$ called the {\em duality} functor, moreover, 
from $\sigma(h)=h$ and $\sigma(z)=z$ it follows that $\mathrm{supp}(M^{\star})=M$ for all $M\in\mathcal{O}$.
As simple modules in $\mathcal{O}$ are uniquely determined by their character (in fact, by their highest weight),
it follows that $L(\lambda)^{\star}\cong L(\lambda)$ for all $\lambda\in\mathfrak{h}^*$.

\subsection{Blocks of nonzero central charge and  half-integral weights}\label{s3.6}

Let $\xi\in \mathfrak{h}^*/\mathbb{Z}R$ be of nonzero central charge and assume that 
$\lambda(h)\in\mathbb{Z}+\frac{1}{2}$ for any $\lambda\in \xi$. Note that the dot-action of $W$ preserves
$\xi$. For $\lambda\in\xi$ such that $\lambda(h)\geq -\frac{3}{2}$ denote by $\mathcal{O}[\xi]_{\lambda}$
the Serre subcategory of $\mathcal{O}[\xi]$ generated by $\Delta(\lambda)$ and $\Delta(r\cdot \lambda)$
(explicitly, we have $r\cdot \lambda=-\lambda-3h^{\checkmark}$). Note that $\Delta(\lambda)=\Delta(r\cdot \lambda)$
if $\lambda(h)=-\frac{3}{2}$. For $i\in\mathbb{Z}_+$ denote by $\lambda_i$ the element in $\xi$
such that $\lambda_i(h)=-\frac{3}{2}+i$.

\begin{proposition}\label{prop4}
Let $\xi\in \mathfrak{h}^*/\mathbb{Z}R$ be of nonzero central charge and assume that 
$\lambda(h)\in\mathbb{Z}+\frac{1}{2}$ for any $\lambda\in \xi$. Then we have the following:
\begin{enumerate}[$($i$)$]
\item\label{prop4.1} For $\lambda\in\xi$ the module $\Delta(\lambda)$ is simple if and only if 
$\lambda(h)\leq -\frac{3}{2}$.
\item\label{prop4.2} For each $i\in\mathbb{N}$ we have a non-split
short exact sequence
\begin{displaymath}
0\to \Delta(r\cdot \lambda_i)\to \Delta(\lambda_i) \to L(\lambda_i)\to 0. 
\end{displaymath}
\item\label{prop4.3} We have the decomposition
\begin{displaymath}
\mathcal{O}[\xi]\cong\bigoplus_{i\in\mathbb{Z}_+}\mathcal{O}[\xi]_{\lambda_i}. 
\end{displaymath} 
\item\label{prop4.4} We have  $\mathcal{O}[\xi]_{\lambda_0}\cong \mathbb{C}\text{-}\mathrm{mod}$,
more precisely, the functor defined on objects as $N\mapsto N_{\lambda}$ and on morphisms in the 
obvious way provides an equivalence between $\mathcal{O}[\xi]_{\lambda_0}$ and the
category of finite dimensional complex vector spaces.
\item\label{prop4.5} For $i\in \mathbb{N}$ the category $\mathcal{O}[\xi]_{\lambda_i}$ is 
equivalent to the category of finite dimensional representations over $\mathbb{C}$ of 
the following quiver with relations:
\begin{displaymath}
\xymatrix{\bullet\ar@/^/[rr]^{a}&&\bullet\ar@/^/[ll]^{b}}\qquad ab=0. 
\end{displaymath}
\end{enumerate}
\end{proposition}

Note that the quiver appearing in Proposition~\ref{prop4}\eqref{prop4.5} is exactly the same quiver which 
describes the  regular block of category $\mathcal{O}$ for $\mathfrak{sl}_2$, see \cite[Section~5.3]{Ma}.
Note also that Proposition~\ref{prop4}\eqref{prop4.5} implies that all $\mathcal{O}[\xi]_{\lambda_i}$,
$i\in \mathbb{N}$, are equivalent.

\begin{proof}
The decomposition in claim~\eqref{prop4.3} is again given using the action of the central element $\mathtt{c}$.
Claim~\eqref{prop4.4} is proved by the same  arguments as used in the proof of Proposition~\ref{lem3}.

The module $\Delta(\lambda_0)$ is simple 
by the same  arguments as used in the proof of Proposition~\ref{lem3}. A straightforward computation shows
that $\mathfrak{n}_+(2 \dot{z}f+q^2)v_{\lambda_1}=0$ which implies that 
$\Delta(r\cdot \lambda_1)$ is a submodule of $\Delta(\lambda_1)$. The quotient 
$N:= \Delta(\lambda_1)/\Delta(r\cdot \lambda_1)$ has Gelfand-Kirillov dimension $1$ and hence contains no
subquotients isomorphic to $\Delta(r\cdot \lambda_1)$. As $L(\lambda_1)$ appears with multiplicity one in
$\Delta(\lambda_1)$, it follows that $N\cong L(\lambda_1)$. This proves claims~\eqref{prop4.1} 
and \eqref{prop4.2} for $\lambda_1$.

Let $V$ be the $2$-dimensional simple $\mathfrak{sl}_2$-module. For $i\in\mathbb{N}$ we have exact biadjoint
functors
\begin{displaymath}
\mathcal{O}[\xi]_{\lambda_i}\overset{\mathrm{incl}}{\longrightarrow} 
\mathcal{O}[\xi]\overset{\mathrm{F}_V}{\longrightarrow}\mathcal{O}[\xi]
\overset{\mathrm{proj}}{\longrightarrow}\mathcal{O}[\xi]_{\lambda_{i+1}}
\end{displaymath}
and
\begin{displaymath}
\mathcal{O}[\xi]_{\lambda_{i+1}}\overset{\mathrm{incl}}{\longrightarrow} 
\mathcal{O}[\xi]\overset{\mathrm{F}_V}{\longrightarrow}\mathcal{O}[\xi]
\overset{\mathrm{proj}}{\longrightarrow}\mathcal{O}[\xi]_{\lambda_{i}}.
\end{displaymath}
The character argument gives that they send Verma modules to Verma modules which implies that they induce
mutually inverse equivalences between $\mathcal{O}[\xi]_{\lambda_i}$ and $\mathcal{O}[\xi]_{\lambda_{i+1}}$.
This proves the first part of claim~\eqref{prop4.5}, moreover, claims~\eqref{prop4.1} 
and \eqref{prop4.2} now follow in the general case from the already checked case of $\lambda_1$.

It remains to prove the second part of claim~\eqref{prop4.5} in the case of $\lambda_1$. This is similar to
\cite[Section~5.3]{Ma}. From the proof of Proposition~\ref{lem3} we know that both $\Delta(\lambda_0)$ 
and $\Delta(\lambda_1)$ are projective in $\mathcal{O}$. We have a pair of biadjoint functors 
\begin{displaymath}
\mathrm{F}:\mathcal{O}[\xi]_{\lambda_0}\overset{\mathrm{incl}}{\longrightarrow} 
\mathcal{O}[\xi]\overset{\mathrm{F}_V}{\longrightarrow}\mathcal{O}[\xi]
\overset{\mathrm{proj}}{\longrightarrow}\mathcal{O}[\xi]_{\lambda_{1}}
\end{displaymath}
and
\begin{displaymath}
\mathrm{G}:\mathcal{O}[\xi]_{\lambda_{1}}\overset{\mathrm{incl}}{\longrightarrow} 
\mathcal{O}[\xi]\overset{\mathrm{F}_V}{\longrightarrow}\mathcal{O}[\xi]
\overset{\mathrm{proj}}{\longrightarrow}\mathcal{O}[\xi]_{\lambda_{0}}.
\end{displaymath}
The character argument gives $\mathrm{G}\Delta(\lambda_{1})\cong \mathrm{G}\Delta(r\cdot \lambda_{1})\cong
\Delta(\lambda_{0})$ and hence, by adjunction, we have 
\begin{displaymath}
\dim \mathrm{Hom}_{\mathcal{O}}(\mathrm{F}\Delta(\lambda_{0}),\Delta(\lambda_{1})) =
\dim \mathrm{Hom}_{\mathcal{O}}(\mathrm{F}\Delta(\lambda_{0}),\Delta(r\cdot \lambda_{1})) =1.
\end{displaymath}
This implies that  $\mathrm{F}\Delta(\lambda_{0})$ is the indecomposable projective cover of the simple module 
$\Delta(r\cdot \lambda_{1}))$. Consider some nonzero homomorphism
$a:\mathrm{F}\Delta(\lambda_{0})\to \Delta(\lambda_{1})$ and let $b$ be a nonzero homomorphism
in the other direction (which exists by adjunction). Then it is easy to see that $ab=0$ 
which implies claim~\eqref{prop4.5}.
\end{proof}

\subsection{Blocks of nonzero central charge and integral weights}\label{s3.7}

Let $\xi\in \mathfrak{h}^*/\mathbb{Z}R$ be of nonzero central charge and assume that 
$\lambda(h)\in\mathbb{Z}$ for any $\lambda\in \xi$. Note that the action of $W$ preserves
$\xi$. For $\lambda\in\xi$ such that $\lambda(h)> -\frac{3}{2}$ denote by $\mathcal{O}[\xi]_{\lambda}$
the Serre subcategory of $\mathcal{O}[\xi]$ generated by $\Delta(\lambda)$ and $\Delta(r\cdot \lambda)$
(explicitly, we have $r\cdot \lambda=-\lambda-3h^{\checkmark}$). For $i\in\mathbb{Z}_+$ denote by $\lambda_i$ 
the element in $\xi$ such that $\lambda_i(h)=-1+i$.

\begin{proposition}\label{prop5}
Let $\xi\in \mathfrak{h}^*/\mathbb{Z}R$ be of nonzero central charge and assume that 
$\lambda(h)\in\mathbb{Z}$ for any $\lambda\in \xi$. Then we have the following:
\begin{enumerate}[$($i$)$]
\item\label{prop5.1} The module $\Delta(\lambda)$ is simple for each $\lambda\in\xi$.
\item\label{prop5.2} We have the decomposition
\begin{displaymath}
\mathcal{O}[\xi]\cong\bigoplus_{i\in\mathbb{Z}_+}\mathcal{O}[\xi]_{\lambda_i}. 
\end{displaymath} 
\item\label{prop5.3} We have  $\mathcal{O}[\xi]_{\lambda_i}\cong \mathbb{C}\oplus\mathbb{C}\text{-}\mathrm{mod}$
for all $i\in\mathbb{Z}_+$.
\end{enumerate}
\end{proposition}

\begin{proof}
The proof is similar to that of Propositions~\ref{lem3} and \ref{prop4}. The decomposition in claim~\eqref{prop5.2} 
is again  given using the action of the central element $\mathtt{c}$.

That $\Delta(r\cdot \lambda_i)$ is simple for each $i\in\mathbb{Z}_+$ is proved similarly to the analogous statement
in Proposition~\ref{lem3}. That $\Delta(\lambda_0)$ is simple follows from the observation that, on the one
hand, $\dim\Delta(\lambda_0)_{\lambda_0-h^{\checkmark}}=1$ but, on the other hand, the element 
$qv_{\lambda_0}$ does not satisfy $\mathfrak{n}_+ qv_{\lambda_0}=0$ since 
$pqv_{\lambda_0}=\lambda_0(z)v_{\lambda_0}\neq 0$. This implies that $\Delta(\lambda_0)$ is a simple
projective module in $\mathcal{O}[\xi]_{\lambda_0}$. In particular, 
$\mathrm{Ext}_{\mathcal{O}}^1(\Delta(\lambda_0),\Delta(r\cdot \lambda_0))=0$. Applying $\star$ we also
get $\mathrm{Ext}_{\mathcal{O}}^1(\Delta(r\cdot \lambda_0),\Delta(\lambda_0))=0$. This implies that 
$\Delta(r\cdot \lambda_0)$ is also a simple projective module in $\mathcal{O}[\xi]_{\lambda_0}$ and hence
$\mathcal{O}[\xi]_{\lambda_0}\cong \mathbb{C}\oplus\mathbb{C}\text{-}\mathrm{mod}$.

Now, similarly to the proof of Propositions~\ref{prop4}, using projective functors one shows that 
$\mathcal{O}[\xi]_{\lambda_i}\cong \mathcal{O}[\xi]_{\lambda_j}$ for all $i,j\in \mathbb{Z}_+$. Claims
\eqref{prop5.1} and \eqref{prop5.3} follow.
\end{proof}

Propositions~\ref{prop4} and \ref{prop5} completely describe all blocks of $\mathcal{O}$ with nonzero
central charge, in particular, we see that all indecomposable such blocks are equivalent to 
indecomposable blocks of $\mathcal{O}$ for $\mathfrak{sl}_2$. As we will see in the next section, for
zero central charge the situation is quite different.

\subsection{Tensor product realization}\label{s3.9}

For $\dot{z}\in\mathbb{C}\setminus\{0\}$ consider the algebras $A_{\dot{z}}:=U(\mathfrak{s})/(z-\dot{z})$
and $B_{\dot{z}}:=U(\mathfrak{i})/(z-\dot{z})$. Note that $B_{\dot{z}}$ is isomorphic to the first Weyl algebra, in 
particular, $B_{\dot{z}}$ is a simple algebra. Following \cite[Theorem~1]{LMZ1} define the homomorphism
$\Phi:A_{\dot{z}}\to B_{\dot{z}}$ as follows:
\begin{displaymath}
\Phi:e\to \frac{p^2}{2\dot{z}},\qquad 
\Phi:f\to -\frac{q^2}{2\dot{z}},\qquad 
\Phi:h\to -\frac{qp}{\dot{z}}-\frac{1}{2}.
\end{displaymath}
Consider the (unique) ``highest weight'' $B_{\dot{z}}$-module $\mathbf{M}:=B_{\dot{z}}/B_{\dot{z}}p$. This
is a simple $B_{\dot{z}}$-module. Pulling back via $\Phi$, the module $\mathbf{M}$ becomes a simple highest weight 
$U$-module with highest weight $-\frac{1}{2}$ and central charge $\dot{z}$.

Let $\mathcal{O}^{(\mathfrak{sl}_2)}$ denote the usual category $\mathcal{O}$ for $\mathfrak{sl}_2$
(see e.g. \cite[Chapter~5]{Ma}). We may regard $\mathcal{O}^{(\mathfrak{sl}_2)}$ as a full subcategory
of $\mathcal{O}$ via the quotient map $\mathfrak{s}\tto \mathfrak{sl}_2$.

Let $\mathcal{O}[\dot{z}]$ denote the full subcategory of $\mathcal{O}$ consisting of all modules with 
central charge $\dot{z}$.

\begin{proposition}\label{prop721}
Tensoring with $\mathbf{M}$ and using the usual comultiplication in $U$ defines a functor
\begin{displaymath}
\mathbf{M}\otimes{}_-: \mathcal{O}^{(\mathfrak{sl}_2)}\to \mathcal{O}[\dot{z}].
\end{displaymath}
Moreover, this functor is an equivalence of categories which sends Verma $\mathfrak{sl}_2$-modules to Verma $U$-modules.
\end{proposition}

\begin{proof}
Functoriality and exactness of $\mathbf{M}\otimes{}_-$ are clear.
That this functor sends simple modules to simple modules follows from \cite[Theorem~7]{LZ},
see also \cite[Theorem~3]{LMZ1}.
That it sends Verma modules to Verma modules follows from \cite[Theorem~2]{LMZ1}.
In particular, it sends projective Verma modules to projective Verma modules. 
Because of the associativity of the tensor product, the functor $\mathbf{M}\otimes{}_-$ commutes with 
projective functors. As each projective functor is biadjoint to a projective functor, each projective functor sends
projective modules to projective modules. Applying projective functors to projective Verma modules
produces all indecomposable projectives both in $\mathcal{O}^{(\mathfrak{sl}_2)}$ and in $\mathcal{O}[\dot{z}]$.
It follows that $\mathbf{M}\otimes{}_-$ sends projective modules to projective modules.

Using the usual inductive argument and tensoring with the simple $2$-dimensional $\mathfrak{sl}_2$-module
one now verifies that $\mathbf{M}\otimes{}_-$ sends indecomposable projective modules to 
indecomposable projective modules. Moreover, by construction this functor clearly does not annihilate
any homomorphisms. Now the statement of the proposition follows by comparing the descriptions of
$\mathcal{O}^{(\mathfrak{sl}_2)}$ (see e.g. \cite[Chapter~5]{Ma}) with the description on 
$\mathcal{O}[\dot{z}]$ obtained above.
\end{proof}

\section{Blocks with zero central charge}\label{s4}

\subsection{Indecomposability}\label{s4.1}

As the first step towards understanding the structure of blocks of zero central charge we prove the following:

\begin{lemma}\label{lem11}
Let $\xi\in \mathfrak{h}^*/\mathbb{Z}R$ be of zero central charge.
\begin{enumerate}[$($i$)$] 
\item\label{lem11.1} There is an inclusion $\Delta(\lambda-h^{\checkmark})\hookrightarrow
\Delta(\lambda)$ for any $\lambda\in\xi$.
\item\label{lem11.2} $\mathcal{O}[\xi]$ is an indecomposable category.
\end{enumerate}
\end{lemma}

\begin{proof}
We obviously have $e\cdot qv_{\lambda}=0$ and, moreover, $p\cdot qv_{\lambda}=zv_{\lambda}=0$ (as the central charge
is zero). Therefore mapping $v_{\lambda-h^{\checkmark}}$ to $qv_{\lambda}$ extends to a nonzero homomorphism
from $\Delta(\lambda-h^{\checkmark})$ to $\Delta(\lambda)$, which is necessarily injective (see Subsection~\ref{s3.2}).
This proves claim~\eqref{lem11.1}.

As $\Delta(\lambda)$ is indecomposable, from claim~\eqref{lem11.1} it follows that $L(\lambda)$ and 
$L(\lambda-h^{\checkmark})$ belong to the same indecomposable direct summand of $\mathcal{O}[\xi]$
for any $\lambda\in\xi$. Claim~\eqref{lem11.2} follows.
\end{proof}

\subsection{Truncated categories}\label{s4.2}

Let $\xi\in \mathfrak{h}^*/\mathbb{Z}R$ be of zero central charge and $\lambda\in\xi$. Denote by 
$\mathcal{O}[\xi,\lambda]$ the full subcategory of $\mathcal{O}[\xi]$ consisting of all modules $M$
such that $M_{\lambda+ih^{\checkmark}}=0$ for all $i\in\mathbb{N}$. Alternatively, 
$\mathcal{O}[\xi,\lambda]$ is the Serre subcategory of $\mathcal{O}[\xi]$ generated by 
modules $L(\lambda-ih^{\checkmark})$, $i\in\mathbb{Z}_+$. Directly from the definition we have
$\mathcal{O}[\xi,\lambda]\hookrightarrow \mathcal{O}[\xi,\lambda+h^{\checkmark}]$ for every
$\lambda\in\xi$ and $\mathcal{O}[\xi]$ is exactly the inductive (direct) limit of this directed system
of categories. Note that the duality $\star$ preserves each $\mathcal{O}[\xi,\lambda]$ while 
projective functors do not preserve these truncated subcategories. The idea of definition of 
the categories $\mathcal{O}[\xi,\lambda]$ is taken from \cite{DGK,RCW} 
(where it is applied to category $\mathcal{O}$ for Kac-Moody Lie algebras, see also \cite{MP,FKM}).

Denote by $\mathcal{F}(\Delta)$ the full subcategory of $\mathcal{O}[\xi]$ consisting of all modules
having a {\em Verma flag}, that is a filtration whose subquotients are isomorphic to Verma modules.
The reason for introducing $\mathcal{O}[\xi,\lambda]$ is the fact that $\mathcal{O}[\xi]$ does not have
nonzero projective objects, while for $\mathcal{O}[\xi,\lambda]$ we have the following:

\begin{proposition}\label{prop12}
Let $\xi\in \mathfrak{h}^*/\mathbb{Z}R$ be of zero central charge, $\lambda\in\xi$
and $i\in\mathbb{Z}_+$.
\begin{enumerate}[$($i$)$]
\item\label{prop12.1} The module $L(\lambda-ih^{\checkmark})$ is a quotient of a
unique, up to isomorphism, indecomposable projective object $P^{(\lambda)}(\lambda-ih^{\checkmark} )$
in $\mathcal{O}[\xi,\lambda]$.
\item\label{prop12.2} We have $P^{(\lambda)}(\lambda-ih^{\checkmark})\tto 
\Delta(\lambda-ih^{\checkmark})$ and the kernel $K$ of this
epimorphism has a Verma flag. Moreover, the only Verma modules occurring as subquotients
in a Verma flag of $K$ are $\Delta(\lambda-jh^{\checkmark})$ where $j<i$.
\end{enumerate}
\end{proposition}

\begin{proof}
This is similar to \cite{BGG}. Set $\mu:=\lambda-ih^{\checkmark}$. Denote by
$I$ the left ideal in $U$ generated by $h-\mu(h)$, $z$ and $U_j$ for all $j>i$. Then 
for the $U$-module $P:=U/I$ we have
$P\in \mathcal{O}[\xi,\lambda]$, moreover, we have
$\mathrm{Hom}_{\mathcal{O}[\xi,\lambda]}(P,N)=N_{\mu}$ for any $N\in \mathcal{O}[\xi,\lambda]$.
As $N\mapsto N_{\mu}$ is an exact functor, the module $P$ is projective. As 
$\mathrm{Hom}_{\mathcal{O}[\xi,\lambda]}(P,L_{\mu})=
L(\mu)_{\mu}\neq 0$, the module $P$ has an indecomposable direct summand which surjects onto $L(\mu)$.
This proves claim~\eqref{prop12.1}.

It follows from the PBW theorem that the module $P$ constructed above has a Verma flag and
the only Verma modules occurring as subquotients in any Verma flag of $P$ are $\Delta(\lambda-jh^{\checkmark})$ 
where $j\leq i$. As in \cite[Proposition~2]{BGG}, we have that $\mathcal{F}[\Delta]$ is closed under taking 
direct summands. Claim~\eqref{prop12.2} follows. 
\end{proof}

Proposition~\ref{prop12} says that $\mathcal{O}[\xi,\lambda]$ is a {\em highest weight category} in the
sense of \cite{CPS}. Simple modules in this category are indexed by $\lambda-ih^{\checkmark}$, where
$i\in\mathbb{Z}_+$, with the natural order $\lambda-ih^{\checkmark}>\lambda-jh^{\checkmark}$ if $i<j$. 
In particular, the multiplicity $[P^{(\lambda)}(\mu):\Delta(\nu)]$ of
$\Delta(\nu)$ as a subquotient of  a Verma flag of $P^{(\lambda)}(\mu)$ does not depend on the choice
of this flag. Furthermore, using the duality $\star$ and \cite{Ir} we have the following
{\em BGG-reciprocity}:

\begin{corollary}\label{cor14}
Let $\xi\in \mathfrak{h}^*/\mathbb{Z}R$ be of zero central charge, $\lambda\in\xi$
and $i,j\in\mathbb{Z}_+$. Then
\begin{displaymath}
[P^{(\lambda)}(\lambda-ih^{\checkmark}):\Delta(\lambda-jh^{\checkmark})]=
(\Delta(\lambda-jh^{\checkmark}):L(\lambda-ih^{\checkmark})),
\end{displaymath}
where the latter denotes the composition multiplicity.
\end{corollary}

It is worth pointing out that from Lemma~\ref{lem11}\eqref{lem11.1} it follows that 
each $\Delta(\mu)$, $\mu\in\xi$, has infinite length.

\subsection{Grading}\label{s4.25}

Set $\overline{U}:=U(\overline{\mathfrak{s}})$. The algebra $\overline{U}$ admits a natural 
$\mathbb{Z}$-gra\-ding by setting
\begin{displaymath}
\deg(e)=\deg(f)=\deg(h)=0,\qquad \deg(p)=\deg(q)=1.
\end{displaymath}
Note that for any $\xi\in \mathfrak{h}^*/\mathbb{Z}R$ of zero central charge any object in $\mathcal{O}[\xi]$
is, in fact, a $\overline{U}$-module. This can be used to define the following {\em graded lift} 
$\overline{\mathcal{O}}[\xi]$ of the  category $\mathcal{O}[\xi]$: The category $\overline{\mathcal{O}}[\xi]$
is defined as the full subcategory of the category of $\mathbb{Z}$-graded $\overline{U}$-modules which belong 
to $\mathcal{O}[\xi]$ after forgetting the grading (cf. \cite{CG}). We denote by 
$\Theta_{\xi}:\overline{\mathcal{O}}[\xi]\to\mathcal{O}[\xi]$ the forgetful functor. Morphisms in 
$\overline{\mathcal{O}}[\xi]$ are homogeneous $\overline{U}$-homomorphisms of degree zero. 

From now on by {\em graded} we always mean {\em $\mathbb{Z}$-graded}. A graded vector space $V$ is written as 
\begin{displaymath}
V=\bigoplus_{i\in\mathbb{Z}}V_i.
\end{displaymath}
For $k\in\mathbb{Z}$ we denote by $\langle k\rangle$ the {\em shift of the grading} functor normalized
as follows: $V\langle k\rangle_i:=V_{i-k}$.

An object  $M\in \mathcal{O}[\xi]$ is called {\em gradable} provided that there is 
$\overline{M}\in \overline{\mathcal{O}}[\xi]$ such that $\Theta_{\xi}\,\overline{M}\cong M$.
If $M\in \mathcal{O}[\xi]$ is an $\mathfrak{sl}_2$-module, that is $pM=qM=0$, then $M$ is gradable by
setting, for $i\in\mathbb{Z}$,
\begin{displaymath}
\overline{M}_i:=\begin{cases}M,& i=0;\\0,&i\neq 0.\end{cases}
\end{displaymath}
We will call this $\overline{M}$ the {\em standard graded lift} of $M$.
In particular, all simple objects in $\mathcal{O}[\xi]$ are gradable. Note that a Verma $\overline{U}$-module
is defined as the quotient of $\overline{U}$ modulo a left ideal generated by homogeneous elements. Hence
all Verma $\overline{U}$-modules are gradable. It is easy to check that $M\oplus N$ is gradable if and
only if $M$ and $N$ are. 

In the standard way the duality $\star$ admits a graded lift which we will denote by the same symbol.
We have the following isomorphism of $\mathfrak{sl}_2$-modules: $(M^{\star})_i\cong (M_{-i})^{\star}$.

\subsection{Non-integral blocks}\label{s4.3}

Let $\xi\in \mathfrak{h}^*/\mathbb{Z}R$ be of zero central charge and assume that $\xi$ is {\em non-integral}
in the sense that $\lambda(h)\not\in\mathbb{Z}$ for some (and hence for any) $\lambda\in\xi$. Consider the
following two quivers:
\begin{displaymath}
{}_{\infty}\mathbf{Q}:\qquad\qquad
\xymatrix{
\ldots\ar@/^/[rr]^{a} && \mathtt{2}\ar@/^/[rr]^{a}\ar@/^/[ll]^{b}
&& \mathtt{1}\ar@/^/[rr]^{a}\ar@/^/[ll]^{b}&& \mathtt{0}\ar@/^/[ll]^{b}
}
\end{displaymath}
and
\begin{displaymath}
{}_{\infty}\mathbf{Q}_{\infty}:\quad\,\,
\xymatrix{
\ldots\ar@/^/[rr]^{a} && \text{-}\mathtt{1}\ar@/^/[rr]^{a}\ar@/^/[ll]^{b}
&& \mathtt{0}\ar@/^/[rr]^{a}\ar@/^/[ll]^{b}&& \mathtt{1}\ar@/^/[ll]^{b}
\ar@/^/[rr]^{a}&& \ldots\ar@/^/[ll]^{b}
}
\end{displaymath}
with imposed commutativity relation $ab=ba$ (which includes $ab=0$ for the vertex $\mathtt{0}$ in the
quiver ${}_{\infty}\mathbf{Q}$).
We denote by ${}_{\infty}\mathbf{Q}\text{-}\mathrm{lfmod}$ the category of locally finite dimensional
${}_{\infty}\mathbf{Q}$-modules (in which $ab=ba$ as above). We also denote by 
${}_{\infty}\mathbf{Q}_{\infty}^+\text{-}\mathrm{lfmod}$ the category of locally finite dimensional
${}_{\infty}\mathbf{Q}_{\infty}$-modules  (in which $ab=ba$) that are {\em bounded from the right}, 
that is modules in which $\mathtt{i}$ is represented by the zero vector space for all $i\gg0$.

\begin{theorem}\label{thm15}
Let $\xi\in \mathfrak{h}^*/\mathbb{Z}R$ be non-integral and of zero central charge.
\begin{enumerate}[$($i$)$]
\item\label{thm15.1} For every $\lambda\in\xi$ the category $\mathcal{O}[\xi,\lambda]$
is equivalent to ${}_{\infty}\mathbf{Q}\text{-}\mathrm{lfmod}$.
\item\label{thm15.2} The category $\mathcal{O}[\xi]$
is equivalent to ${}_{\infty}\mathbf{Q}_{\infty}^+\text{-}\mathrm{lfmod}$.
\end{enumerate}
\end{theorem}

\begin{proof}
For $i\in\mathbb{Z}_+$ we assign to the simple object $L(\lambda-ih^{\checkmark})$ the vertex $\mathtt{i}$ 
in the quiver ${}_{\infty}\mathbf{Q}$. First we claim that for all $i,j\in \mathbb{Z}_+$ such that $i\leq j$
we have
\begin{equation}\label{eq17}
\mathrm{Ext}_{\mathcal{O}}^1(L(\lambda-ih^{\checkmark}),L(\lambda-jh^{\checkmark}))\cong
\begin{cases}
\mathbb{C}, & j=i+1;\\
0, & \text{otherwise}.
\end{cases}
\end{equation}
Indeed, consider a non-split short exact sequence
\begin{displaymath}
0\to  L(\lambda-jh^{\checkmark})\to X\to L(\lambda-ih^{\checkmark})\to 0.
\end{displaymath}
Then $X$ is generated by a highest weight vector of weigh $\lambda-ih^{\checkmark}$ and hence is a quotient
of $\Delta(\lambda-ih^{\checkmark})$. The latter module has simple top. By Lemma~\ref{lem11}\eqref{lem11.1} we 
have an inclusion $\Delta(\lambda-(i+1)h^{\checkmark})\hookrightarrow\Delta(\lambda-ih^{\checkmark})$
and the quotient is simple, by character argument, already as an $\mathfrak{sl}_2$-module, since 
$\xi$ is non-integral. This means that
\begin{displaymath}
\Delta(\lambda-(i+1)h^{\checkmark})\cong \mathrm{Rad}(\Delta(\lambda-ih^{\checkmark})).
\end{displaymath}
Since $\Delta(\lambda-(i+1)h^{\checkmark})$ has simple top $L(\lambda-(i+1)h^{\checkmark})$, we get 
formula \eqref{eq17}.

Using $\star$ and the fact that $\star$ preserves isomorphism classes of simple modules, we get
\begin{displaymath}
\mathrm{Ext}_{\mathcal{O}}^1(L(\lambda-ih^{\checkmark}),L(\lambda-jh^{\checkmark}))\cong 
\mathrm{Ext}_{\mathcal{O}}^1(L(\lambda-jh^{\checkmark}),L(\lambda-ih^{\checkmark}))
\end{displaymath}
which, together with \eqref{eq17}, says that the quiver of the category $\mathcal{O}[\xi,\lambda]$
is exactly the underlying quiver of ${}_{\infty}\mathbf{Q}$.  Note that non-split extensions between 
$L(\lambda-ih^{\checkmark})$ and $L(\lambda-(i+1)h^{\checkmark})$ are given (inside 
$\Delta(\lambda-ih^{\checkmark})$ and $\Delta(\lambda-ih^{\checkmark})^{\star}$, respectively) by the action
of $p$ or $q$, respectively, and we have $pq-qp=z=0$ as we are in the situation of 
zero central charge. This suggests that the relations in the quiver should be commutativity relations.
However, to rigorously prove the letter guess it is convenient to consider the graded lift. 

Define $\overline{\mathcal{O}}[\xi,\lambda]$ as the full subcategory of the category of graded $\overline{U}$-modules which belong  to $\mathcal{O}[\xi,\lambda]$ after forgetting the grading. We claim that indecomposable
projective modules in $\mathcal{O}[\xi,\lambda]$ are gradable. Indeed, using the original construction of \cite{BGG},
indecomposable projective modules in $\mathcal{O}[\xi,\lambda]$ are direct summand of the following
projective objects (here $k\in\mathbb{Z}_+$):
\begin{displaymath}
P_{(k)}:=\overline{U}/\overline{U}(h-(\lambda-kh^{\checkmark})(h),p^{k+1},e^{\lceil\frac{k+1}{2}\rceil}).
\end{displaymath}
As both $h-(\lambda-kh^{\checkmark})(h)$, $p^{k+1}$ and $e^{\lceil\frac{k+1}{2}\rceil}$ are homogeneous 
elements, it follows that $P_{(k)}$ is gradable.

For $k\in\mathbb{Z}_+$ let $\overline{P}(k)$ denote the indecomposable graded projective such that 
$\overline{P}(k)\tto \overline{L}(\lambda-kh^{\checkmark})$. Set $\overline{I}(k):=\overline{P}(k)^{\star}$.
Then $\overline{I}(k)$ is the indecomposable graded injective envelope of $\overline{L}(\lambda-kh^{\checkmark})$.
The full subcategory of $\mathcal{O}[\xi,\lambda]$ with objects $\Theta_{\xi}\overline{P}(k)$, $k\in\mathbb{Z}_+$,
is thus graded, which implies that the quiver of $\mathcal{O}[\xi,\lambda]$ is graded as well. 
In particular, the whole highest weight structure on $\mathcal{O}[\xi,\lambda]$ is gradable (in the sense of
\cite{MO}).

\begin{lemma}\label{lem227}
For every $k\in\mathbb{N}$ there are unique (up to a nonzero scalar) nonzero homomorphism as follows:
\begin{enumerate}[$($a$)$]
\item\label{lem227.1} $\overline{P}(k\pm 1)\langle -1\rangle\to \overline{P}(k)$;
\item\label{lem227.2} $\overline{P}(k)\langle -2\rangle\to \overline{P}(k)$.
\end{enumerate}
\end{lemma}

\begin{proof}
With respect to our grading we have $\overline{U}_0=U(\mathfrak{sl}_2)$ and
$\overline{U}_1=V\otimes U(\mathfrak{sl_2})$ where  $V=\mathfrak{i}/\mathbb{C}z$ is the $2$-dimensional 
$\mathfrak{sl}_2$-module spanned by $p$ and $q$. Clearly, we have
$\overline{U}\otimes_{U(\mathfrak{sl}_2)}\overline{L}(\lambda-kh^{\checkmark})\tto \overline{P}(k)$.
This implies that $V\otimes \overline{L}(\lambda-kh^{\checkmark})\tto \overline{P}(k)_1$.
A character argument combined with our computation of extensions above gives
\begin{displaymath}
V\otimes \overline{L}(\lambda-kh^{\checkmark})\cong  
\overline{L}(\lambda-(k-1)h^{\checkmark})\oplus \overline{L}(\lambda-(k+1)h^{\checkmark}). 
\end{displaymath}
As $\overline{P}(k)\tto \overline{\Delta}(\lambda-kh^{\checkmark})$ and
$\overline{\Delta}(\lambda-(k-1)h^{\checkmark})\langle-1\rangle
\hookrightarrow\overline{\Delta}(\lambda-kh^{\checkmark})$
by Lemma~\ref{lem11}\eqref{lem11.1}, we get that $\overline{P}(k)_1$ 
contains $\overline{L}(\lambda-(k-1)h^{\checkmark})\langle-1\rangle$.
Using $\star$ we get that $\overline{P}(k)_1$ contains $\overline{L}(\lambda-(k+1)h^{\checkmark})\langle-1\rangle$.
Claim~\eqref{lem227.1} follows.

Consider a Verma flag of $\overline{P}(k)$. It contains the subquotient $\overline{\Delta}(\lambda-kh^{\checkmark})$
and, clearly, $[\overline{\Delta}(\lambda-kh^{\checkmark}):\overline{L}(\lambda-kh^{\checkmark})\langle-2\rangle]=0$.
From claim~\eqref{lem227.1} we also have the subquotient 
$\overline{\Delta}(\lambda-(k-1)h^{\checkmark})\langle-1\rangle$ and the multiplicity 
$[\overline{\Delta}(\lambda-(k-1)h^{\checkmark})\langle-1\rangle:
\overline{L}(\lambda- kh^{\checkmark})\langle-2\rangle]=1$. Any other Verma subquotients are of the form
$\overline{\Delta}(\lambda-jh^{\checkmark})\langle -i\rangle$ where $j<k$ and $i\geq 2$. For these subquotients
we have $[\overline{\Delta}(\lambda-jh^{\checkmark})\langle -i\rangle:
\overline{L}(\lambda- kh^{\checkmark})\langle-2\rangle]=0$. Claim~\eqref{lem227.2} follows. 
\end{proof}

From Lemma~\ref{lem227}\eqref{lem227.1} we get that the grading on $\overline{\mathcal{O}}[\xi,\lambda]$ 
agrees with the usual grading on ${}_{\infty}\mathbf{Q}$ in which  each arrow has degree one.
For $k\in\mathbb{Z}_+$ fix some nonzero homomorphisms
\begin{displaymath}
\varphi_k: \overline{P}(k)\langle -1\rangle\to \overline{P}(k+1)\quad\text{ and }\quad
\psi_k: \overline{P}(k+1)\langle -1\rangle\to \overline{P}(k).
\end{displaymath}
From Lemma~\ref{lem227}\eqref{lem227.2}, for $k>0$ the homomorphisms
$\psi_k\langle -1\rangle\circ \varphi_k$ and $\varphi_{k-1}\langle -1\rangle\circ \psi_{k-1}$ 
are linearly dependent. Note that $\psi_0\langle -1\rangle\circ \varphi_0=0$ as
$\overline{P}(0)$ is a Verma module and hence
$[\overline{P}(0):\overline{L}(0)\langle-2\rangle]=0$.

\begin{lemma}\label{lem228}
For every $k\in\mathbb{N}$ there is a nonzero scalar $a_k\in\mathbb{C}$ such that 
$\psi_k\langle -1\rangle\circ \varphi_k-a_k\varphi_{k-1}\langle -1\rangle\circ \psi_{k-1}=0$.
\end{lemma}

\begin{proof}
From Lemma~\ref{lem227}\eqref{lem227.2} we have that  there is a unique (up to scalar) nonzero morphism
from $\overline{P}(k)$ to $\overline{I}(k)\langle -2\rangle$. Let $N$ be its image. The statement of the lemma
is equivalent to saying that we have the following isomorphism of the first graded component: 
$N_{1}\cong \overline{P}(k)_1$. From the proof of Lemma~\ref{lem227} we know that
\begin{displaymath}
\overline{P}(k)_1\cong \overline{L}(\lambda-(k-1)h^{\checkmark})\langle-1\rangle\oplus
\overline{L}(\lambda-(k+1)h^{\checkmark})\langle-1\rangle.
\end{displaymath}
Therefore, replacing $\lambda$ by $\lambda-(k-1)h^{\checkmark}$, we may assume $k=1$.
In this case we have that $\overline{P}(1)\cong P_{(1)}$ so we identify these two modules.
This allows us to do the following explicit computations (in which we identify elements of $U$
with their images in the corresponding modules).

Denote by $X$ the quotient of $\overline{P}(1)$ by the submodule $\overline{P}(1)_3+\overline{P}(1)_4+\dots$
and by $Y$ the quotient of $\overline{P}(1)$ by the submodule $\overline{P}(1)_2+\overline{P}(1)_3+\dots$.
The submodule $L(\lambda-2h^{\checkmark})\langle-1\rangle$ of $Y$ is generated by the highest weight element
$w_1:=q-\frac{1}{\lambda(h)}fp$ (note that $\lambda(h)\neq 0$ as we are in the situation of a non-integral block).
The submodule $L(\lambda)\langle-1\rangle$ of $Y$ is generated by the highest weight element
$w_2:=p$. Let $w'_1$ and $w'_2$ be some preimages in $X$ of $w_1$ and $w_2$, respectively.
Then we have $qw'_2=pq$ and also $pw'_1=(1+\frac{1}{\lambda(h)})pq$ in $X$. Again note that
$1+\frac{1}{\lambda(h)}\neq 0$ as we are in the situation of a non-integral block. 
The element $pq$ is exactly the highest weight element of the submodule $L(\lambda-h^{\checkmark})\langle-2\rangle$
in $X$. So, we have just proved that the action of $U$ on both composition subquotients of 
$\overline{P}(1)_1$ leads to a nonzero contribution to $L(\lambda-h^{\checkmark})\langle-2\rangle$.
The implies $N_{1}\cong \overline{P}(k)_1$ and the claim of the lemma follows.
\end{proof}

From Lemma~\ref{lem228} it follows that, rescaling the $\varphi_{k}$'s, if necessary, we may assume that 
$\psi_k\langle -1\rangle\circ \varphi_k=\varphi_{k-1}\langle -1\rangle\circ \psi_{k-1}$.
This means that the quiver of $\mathcal{O}[\xi,\lambda]$ is a quotient of 
${}_{\infty}\mathbf{Q}$. To prove that they coincide we have just to compare the Cartan
data of both categories.

It is easy to check that the category ${}_{\infty}\mathbf{Q}\text{-}\mathrm{lfmod}$ is a highest weight
category with respect to the order $\dots<\mathtt{2}<\mathtt{1}<\mathtt{0}$, with standard modules
having the following form:
\begin{displaymath}
\xymatrix{ 
\ldots\ar@/^/[r]^0&\mathbb{C}\ar@/^/[r]^0\ar@/^/[l]^{\mathrm{id}}&\mathbb{C}\ar@/^/[r]^0\ar@/^/[l]^{\mathrm{id}}&
0\ar@/^/[r]^0\ar@/^/[l]^0&\ldots\ar@/^/[r]^0\ar@/^/[l]^0&0\ar@/^/[r]^0\ar@/^/[l]^0&0\ar@/^/[l]^0
}
\end{displaymath}
Note that the multiplicities of simple subquotients in this module are the same as the corresponding
multiplicities of simple subquotients in $\Delta(\lambda-ih^{\checkmark})$ (under our identification of 
$L(\lambda-jh^{\checkmark})$ with $\mathtt{j}$). From the BGG reciprocity we get that the characters
of indecomposable projective modules in ${}_{\infty}\mathbf{Q}\text{-}\mathrm{lfmod}$ and
$\mathcal{O}[\xi,\lambda]$ match. This implies claim~\eqref{thm15.1}. 
Claim~\eqref{thm15.2} follows from claim~\eqref{thm15.1} by taking
the direct limit.
\end{proof}

\subsection{Finite dimensional part of $\mathcal{O}$}\label{s4.4}

Let $\xi\in \mathfrak{h}^*/\mathbb{Z}R$ be of zero central charge and {\em integral} in the sense that 
$\lambda(h)\in\mathbb{Z}$ for some (and hence for all) $\lambda\in\xi$. 

Denote by $\mathcal{O}^{f}$ the full subcategory of $\mathcal{O}$ consisting of all finite-di\-men\-si\-onal
modules in $\mathcal{O}$. Simple finite dimensional $\mathfrak{s}$-modules are exactly simple 
finite dimensional $\mathfrak{sl}_2$-modules. For $i\in\mathbb{Z}_+$ we denote by $\lambda_i$ the highest
weight of the simple $i+1$-dimensional $\mathfrak{s}$-module.
The category $\mathcal{O}^{f}$ is a subcategory of the integral
block $\mathcal{O}[\xi]$ of zero central charge. Namely, it is the Serre subcategory generated by all
$L(\lambda_i)$, $i\in\mathbb{Z}_+$. 

Consider the following quiver:
\begin{displaymath}
\mathbf{Q}_{\infty}:\quad\,\,
\xymatrix{
\mathtt{0}\ar@/^/[rr]^{a}&& \mathtt{1}\ar@/^/[ll]^{b}
\ar@/^/[rr]^{a}&& \mathtt{2}\ar@/^/[ll]^{b}
\ar@/^/[rr]^{a}&& \ldots\ar@/^/[ll]^{b}
}
\end{displaymath}
with imposed commutativity relation $ab=ba$ (which includes the relation $ba=0$ for the vertex $\mathtt{0}$).
We denote by $\mathbf{Q}_{\infty}\text{-}\mathrm{fmod}$ the category of finite dimensional
$\mathbf{Q}_{\infty}$-modules (in which $ab=ba$ as above) that is modules in which each $\mathtt{i}$ 
is represented by a finite dimensional vector space and these vector spaces are zero for all but finitely
many $\mathtt{i}$.

\begin{theorem}\label{thm301}
The categories $\mathcal{O}^{f}$ and $\mathbf{Q}_{\infty}\text{-}\mathrm{fmod}$ are equivalent.
\end{theorem}

\begin{proof}
We use grading similarly to the proof of Theorem~\ref{thm15}.
Let $\mathcal{X}$ denote the category of all graded $\overline{U}$-modules with finite dimensional graded
components. Let $\mathcal{X}^-$ denote the full subcategory of $\mathcal{X}$ consisting of all $M$ 
satisfying the condition $M_i=0$ for all $i\ll 0$. Consider $U(\mathfrak{sl}_2)$ as a graded 
algebra concentrated in degree zero. Let $\mathcal{Y}$ denote the category of all graded 
$U(\mathfrak{sl}_2)$-modules with finite dimensional graded components. Let $\mathcal{Y}^-$ denote the 
full subcategory of $\mathcal{Y}$ consisting of all $M$ satisfying the condition $M_i=0$ for all $i\ll 0$. 
We have the usual exact restriction functor
$\mathrm{Res}^{\overline{\mathfrak{s}}}_{\mathfrak{sl}_2}: \mathcal{X}^-\to \mathcal{Y}^-$.
As $\overline{U}$ is concentrated in non-negative degrees, the right adjoint of
$\mathrm{Res}^{\overline{\mathfrak{s}}}_{\mathfrak{sl}_2}$ maps $\mathcal{Y}^-$ to $\mathcal{X}^-$:
\begin{displaymath}
\mathrm{Ind}^{\overline{\mathfrak{s}}}_{\mathfrak{sl}_2}=\overline{U}\otimes_{U(\mathfrak{sl}_2)}{}_-: 
\mathcal{Y}^-\to \mathcal{X}^-.  
\end{displaymath}
Being the right adjoint of an exact functor, $\overline{U}\otimes_{U(\mathfrak{sl}_2)}{}_-$ maps projective
modules to projective modules. It follows that $\overline{P}(\lambda_i):=
\overline{U}\otimes_{U(\mathfrak{sl}_2)}\overline{L}(\lambda_i)$
is the indecomposable projective cover of $\overline{L}(\lambda_i)$.

Note that $\overline{U}\cong \mathbb{C}[p,q]\otimes U(\mathfrak{sl}_2)$ and for $j\in\mathbb{Z}_+$ the space 
of homogeneous  polynomials in $\mathbb{C}[p,q]$ of degree $j$ is a simple $j+1$-dimensional 
$\mathfrak{sl}_2$-module under the adjoint action. Therefore, as ungraded $\mathfrak{sl}_2$-module, we have
\begin{equation}\label{eq305}
\overline{P}(\lambda_i)_j\cong \overline{L}(\lambda_j)\otimes \overline{L}(\lambda_i).
\end{equation}
In particular, using the classical Clebsch-Gordon rule for $\mathfrak{sl}_2$, see e.g. \cite[Theorem~1.39]{Ma},
we have:
\begin{displaymath}
\overline{P}(\lambda_i)_1\cong 
\begin{cases}
\overline{L}(\lambda_1), & i=0;\\
\overline{L}(\lambda_{i-1})\oplus \overline{L}(\lambda_{i+1}), & i>0.
\end{cases}
\end{displaymath}
It follows that the underlying quiver of $\mathcal{O}^{f}$ is exactly $\mathbf{Q}_{\infty}$.

As $[\overline{P}(\lambda_0)_2:\overline{L}(\lambda_0)\langle -2\rangle]=0$, we get the relation $ba=0$.
As we have $[\overline{P}(\lambda_i)_2:\overline{L}(\lambda_i)\langle -2\rangle]=1$ for $i>0$, 
we get linear dependence of
$ab$ and $ba$ at each $\mathtt{i}$ for $i>0$. A similar computation as in the proof of  
Theorem~\ref{thm15} implies that after a rescaling this reduces to commutativity relation.
The statement is completed by comparing the Cartan data for $\mathcal{O}^{f}$ (which is computed using
\eqref{eq305} and \cite[Theorem~1.39]{Ma}) and that for $\mathbf{Q}_{\infty}\text{-}\mathrm{fmod}$
(which is a straightforward computation). The claim follows.
\end{proof}

\begin{remark}\label{rem302}
{\em 
For $n\in\mathbb{N}$ let $\mathcal{X}_n$ denote the Serre subcategory in the category
$\mathbf{Q}_{\infty}\text{-}\mathrm{fmod}$ generated by all simple modules corresponding to 
$\mathtt{i}$ for $i\leq n$. From Theorem~\ref{thm301} it follows that
$\mathcal{X}_n$ is equivalent to the category of modules over the following quiver:
\begin{displaymath}
\xymatrix{
\mathtt{0}\ar@/^/[rr]^{a}&& \mathtt{1}\ar@/^/[ll]^{b}
\ar@/^/[rr]^{a}&& \mathtt{2}\ar@/^/[ll]^{b}
\ar@/^/[rr]^{a}&& \ldots\ar@/^/[ll]^{b}\ar@/^/[rr]^{a}&& \ar@/^/[ll]^{b}\mathtt{n}
}
\end{displaymath}
with imposed commutativity relation $ab=ba$ (which includes $ba=0$ for the vertex $\mathtt{0}$
and $ab=0$ for the vertex $\mathtt{n}$).
The path algebra of this quiver is known as the {\em preprojective algebra} of type $A$ as defined in 
\cite{GP}. In particular, this algebra has wild representation type for $n>4$, see \cite[Page~2626]{BES}
(note that our numbering of simples starts with $0$). This agrees with the main result of \cite{Mak}
and implies that the main result in \cite{Wu} is not complete.
}
\end{remark}

\subsection{Integral block}\label{s4.5}

Let $\xi\in \mathfrak{h}^*/\mathbb{Z}R$ be of zero central charge and {\em integral} in the sense that 
$\lambda(h)\in\mathbb{Z}$ for some (and hence for all) $\lambda\in\xi$. Consider the following quiver
which we call $\Gamma$:
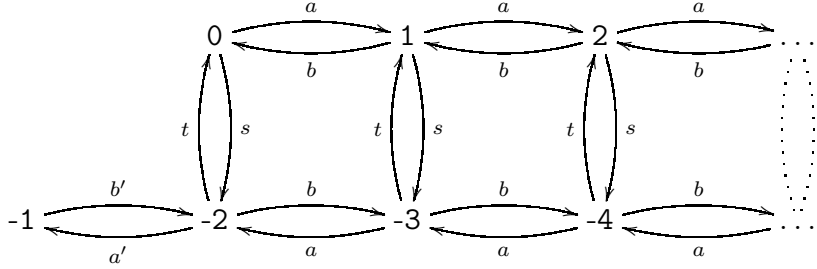
\begin{displaymath}
\xymatrix{
&&\mathtt{0}\ar@/^/[rr]^{a}\ar@/^/[dd]^{s}
&&\mathtt{1}\ar@/^/[rr]^{a}\ar@/^/[ll]^{b}\ar@/^/[dd]^{s}
&&\mathtt{2}\ar@/^/[rr]^{a}\ar@/^/[ll]^{b}\ar@/^/[dd]^{s}&&\dots\ar@/^/[ll]^{b}\ar@/^/@{.}[dd]\\
&&&&&&&&\\
\text{-}\mathtt{1}\ar@/^/[rr]^{b'}&&\text{-}\mathtt{2}\ar@/^/[ll]^{a'}\ar@/^/[uu]^{t}\ar@/^/[rr]^{b}&&
\text{-}\mathtt{3}\ar@/^/[ll]^{a}\ar@/^/[rr]^{b}\ar@/^/[uu]^{t}
&&\text{-}\mathtt{4}\ar@/^/[ll]^{a}\ar@/^/[rr]^{b}\ar@/^/[uu]^{t}&&\dots\ar@/^/[ll]^{a}\ar@/^/@{.}[uu]
}
\end{displaymath}
For $n\in\mathbb{Z}$ we denote by $\Gamma_n$ the full subquiver of $\Gamma$
containing all vertices up to $\mathtt{n}$.
Note that each vertical column is the quiver of the principal block of the category $\mathcal{O}$
for $\mathfrak{sl}_2$, see \cite[Section~5.3]{Ma}. 

\begin{proposition}\label{thm16}
Let $\xi\in \mathfrak{h}^*/\mathbb{Z}R$ be integral and of zero central charge.
\begin{enumerate}[$($i$)$]
\item\label{thm16.1} Let $\lambda\in\xi$ be such that $\lambda(h)=n\in\mathbb{Z}$. Then 
$\Gamma_n$ is the Gabriel quiver of the category $\mathcal{O}[\xi,\lambda]$.
\item\label{thm16.2} The quiver $\Gamma$ is the Gabriel quiver for the category $\mathcal{O}[\xi]$.
\end{enumerate}
\end{proposition}

\begin{proof}
We prove claim~\eqref{thm16.1} and claim~\eqref{thm16.2} is obtained by taking the direct limit.
Let us calculate the first extension space 
between the simple modules $L(ih^{\checkmark})$ and $L(jh^{\checkmark})$, where $i,j\in\mathbb{Z}$
and $i\geq j$. If $i<0$, this is exactly the same calculation as in the proof of Theorem~\ref{thm15}.

Assume $i\geq 0$. Then the module $\Delta((i-1)h^{\checkmark})$ (which has simple top $L((i-1)h^{\checkmark})$)
embeds into $\Delta(ih^{\checkmark})$ and the quotient is the Verma module over $\mathfrak{sl}_2$ with
highest weight $i$. This module has length $2$ with simple socle isomorphic to $L(-(i+2)h^{\checkmark})$.

If $i=0$, then $pfv_{0}=qv_{0}$ which implies that $\Delta(-h^{\checkmark})$ belongs to the submodule
generated by $fv_{0}$. In other words, the radical of $\Delta(0)$ has simple top, namely $L(-2h^{\checkmark})$.

If $i>0$, then the weight of the element $pf^{i+1}v_{ih^{\checkmark}}$ is $-(i+1)h^{\checkmark}$ and 
$L((i-1)h^{\checkmark})_{-(i+1)h^{\checkmark}}=0$ (as the lowest weight of
$L((i-1)h^{\checkmark})$ is exactly $-(i-1)h^{\checkmark}$). This implies that the top of 
$\Delta((i-1)h^{\checkmark})$ is also in the top of the radical of $\Delta(ih^{\checkmark})$ in this case.

The above arguments imply the following for $i,j\in \mathbb{Z}$ with $i>j$:
\begin{displaymath}
\mathrm{Ext}^1_{\mathcal{O}}(L(ih^{\checkmark}),L(jh^{\checkmark}))\cong
\begin{cases}
\mathbb{C}, & i\neq 0, j=i-1;\\
\mathbb{C}, & i\geq 0, j=-i-2;\\
0, & \text{otherwise}.
\end{cases}
\end{displaymath}
Using $\star$ we extend this computation to the case of arbitrary $i,j\in \mathbb{Z}$ 
(by swapping $i$ and $j$ in the left hand side) and see that 
the quiver is the correct one.
\end{proof}

\section{Center of $U$ and annihilators of Verma modules}\label{s5}

\subsection{Intersection of annihilators of Verma modules}\label{s5.1}

For $\lambda\in\mathfrak{h}^*$ set $I_{\lambda}:=\mathrm{Ann}_{U}(\Delta(\lambda))$ and
$J_{\lambda}:=\mathrm{Ann}_{U}(L(\lambda))$. Then both $I_{\lambda}$ and $J_{\lambda}$ are two-sided ideals in $U$
and $I_{\lambda}\subset J_{\lambda}$. In this subsection we observe the following:

\begin{proposition}\label{prop21}
We have $\displaystyle \bigcap_{\lambda\in\mathfrak{h}^*}I_{\lambda}=0$. 
\end{proposition}

\begin{proof}
Fix some PBW basis $\mathbf{B}$ in $U(\mathfrak{n}_-)\otimes U(\mathfrak{n}_+)$. For a nonzero $u\in U$,
use the decomposition  $U\cong U(\mathfrak{n}_-)\otimes U(\mathfrak{n}_+)\otimes U(\mathfrak{h})$ to write 
\begin{displaymath}
u=\sum_{b\in \mathbf{B}} b\otimes x_b, 
\end{displaymath}
where $x_b\in U(\mathfrak{h})$. We have $x_b\neq 0$ for finitely many $b$. Each
$\lambda\in\mathfrak{h}^*$ corresponds naturally to a unique algebra homomorphism 
\begin{equation}\label{eqb1}
\pi_{\lambda}:U(\mathfrak{h})\to\mathbb{C}.
\end{equation}
We may choose $\lambda\in\mathfrak{h}^*$ such that the following two conditions are satisfied:
\begin{eqnarray}
\label{eqa1} \pi_{\lambda}(x_b)\neq 0\quad\text{ whenever }\quad x_b\neq 0;\\
\label{eqa2} \lambda(z)\neq 0\quad\text{ and }\quad  \lambda(h)\not\in\frac{1}{2}\mathbb{Z}.
\end{eqnarray}
Let $N$ be a positive integer which is strictly bigger than the total degree of each monomial $b$
for which $x_b\neq 0$. 

Let $I$ be the left ideal in $U$ generated by $h-\lambda(h)$, $z-\lambda(z)$ and $\mathfrak{n}_+^N$.
Consider the corresponding quotient $U/I$ of the left regular $U$-module. Then condition~\eqref{eqa1} and
our choice of $N$ above guarantee that we have $u\cdot (1+I)=u+I\neq 0$ in $U/I$, that is $u\not\in\mathrm{Ann}_U(U/I)$.
Note that $U/I\in\mathcal{O}$ by construction, more precisely, $U/I\in\mathcal{O}[\lambda+\mathbb{Z}R]$.
Now, condition~\eqref{eqa2} says that we are in the situation described in  Proposition~\ref{lem3} and hence
$U/I$ is a direct sum of Verma modules. The claim follows.
\end{proof}

\subsection{Harish-Chandra homomorphism}\label{s5.2}

Following \cite[Section~7.4]{Di}, we have $U_0=U(\mathfrak{h})\oplus (U_0\cap U\mathfrak{n}_+)$
and $U_0\cap U\mathfrak{n}_+$ is a two-sided ideal of $U_0$. Consider the {\em Harish-Chandra homomorphism}
$\varphi:U_0\to U(\mathfrak{h})$ defined as the projection with respect to the above decomposition.

\begin{proposition}\label{prop22}
We have $\varphi(Z(\mathfrak{s}))=\mathbb{C}[z,z(h+\frac{3}{2})^2]$.
\end{proposition}

\begin{proof}
Let $\lambda\in\mathfrak{h}^*$ and $x\in Z(\mathfrak{s})$. Note that $Z(\mathfrak{s})\subset U_0$.
As $\mathfrak{n}_+v_{\lambda}=0$, we have
\begin{displaymath}
 x\cdot v_{\lambda}= \varphi(x)\cdot v_{\lambda}= \pi_{\lambda}(\varphi(x))v_{\lambda},
\end{displaymath}
where $\pi_{\lambda}$ is as in \eqref{eqb1}. Moreover, $x$ acts on $\Delta(\lambda)$ as the scalar
$\pi_{\lambda}(\varphi(x))$. If $x\neq 0$, then, by Proposition~\ref{prop21}, there exists $\lambda$ such 
that $\pi_{\lambda}(\varphi(x))\neq 0$. It follows that $\varphi(x)\neq 0$ and hence the restriction of
$\varphi$ to $Z(\mathfrak{s})$ is injective.

We have $\varphi(z)=z$ and $\varphi(\mathtt{c})=z(h+\frac{3}{2})^2-\frac{1}{4}z$ and hence we have
$\varphi(Z(\mathfrak{s}))\supset\mathbb{C}[z,z(h+\frac{3}{2})^2]$. To complete the proof it is thus
left to show that $\varphi(Z(\mathfrak{s}))\subset\mathbb{C}[z,z(h+\frac{3}{2})^2]$.
 
For $x\in Z(\mathfrak{s})$ consider the polynomial $\varphi(x)$ in $h$ and $z$. Let 
$\xi\in \mathfrak{h}^*/\mathbb{Z}R$ be of zero central charge. From Lemma~\ref{lem11}\eqref{lem11.2} it follows
that the value $\pi_{\lambda}(\varphi(x))$ does not depend on the choice of $\lambda\in \xi$. It follows
that the evaluation of $\varphi(x)$ at $z=0$ is a constant, that is $\varphi(x)=c+zf(h,z)$ for some
$c\in\mathbb{C}$ and $f(h,z)\in U(h,z)$. 

Write $\varphi(x)=c+zf_1(h)+z^2f_2(h)+\dots+z^kf_k(h)$ for some polynomials $f_1(h),\dots,f_k(h)\in \mathbb{C}[h]$.
From Proposition~\ref{prop4}\eqref{prop4.5} it follows that for any $\dot{z}\in\mathbb{C}\setminus\{0\}$
and any $i\in\mathbb{Z}_+$ we have 
\begin{multline*}
c+\dot{z}f_1(i-3/2)+\dot{z}^2f_2(i-3/2)+\dots+\dot{z}^kf_k(i-3/2)=\\=
c+\dot{z}f_1(-i-3/2)+\dot{z}^2f_2(-i-3/2)+\dots+\dot{z}^kf_k(-i-3/2).
\end{multline*}
As functions $z,z^2,\dots,z^k$ are linearly independent, we obtain the equalities 
$f_j(-\frac{3}{2}+i)=f_j(-\frac{3}{2}-i)$ for all $j=1,2,\dots,k$ and $i\in\mathbb{Z}_+$. This implies that
$f_j(h)$ is a polynomial in $(h+\frac{3}{2})^2$ for all $j=1,2,\dots,k$. 

Now we claim that $z^i\varphi(x)\in \mathbb{C}[z,z(h+\frac{3}{2})^2]$ for some $i>0$. Indeed, choose $i$
such that for every $j=1,2,\dots,k$ the degree of $f_j$ (as a polynomial in $(h+\frac{3}{2})^2$) does not exceed 
$j+i$. Since we have $\varphi(z)=z$ and also $\varphi(\mathtt{c})=z(h+\frac{3}{2})^2-\frac{1}{4}z$, 
there exists $g(z,\mathtt{c})\in\mathbb{C}[z,\mathtt{c}]$ such that $\varphi(g(z,\mathtt{c}))=z^i\varphi(x)$.
From the injectivity of $\varphi$ it now follows that $z^ix=g(z,\mathtt{c})$. Moving all terms containing $z$ 
to the left, we get $zy=\tilde{g}(\mathtt{c})$ for some $y\in\ Z(\mathfrak{s})$ and some 
$\tilde{g}(\mathtt{c})\in\mathbb{C}[\mathtt{c}]$.

We claim that $y=0$ and $\tilde{g}(\mathtt{c})=0$. Indeed, assume that this is not the case and 
write $v=zy=\tilde{g}(\mathtt{c})$ in the PBW basis
of $U$ with respect to the basis $f,q,h,p,e,z$ of $\mathfrak{s}$. Then, on the one hand, $v$ has nonzero
coefficients only at basis elements containing $z$ (because $v=zy$). It follows that $\tilde{g}$ is not a constant
polynomial, say it has degree $d>0$. But then, on the other hand, $v$ must have a nonzero coefficient at
$f^dp^{2d}$ (since $v=\tilde{g}(\mathtt{c})$), a contradiction.

As $U$ is a domain, the equality $zy=0$ implies $y=0$ which, in turn, means that $z$ divides the polynomial
$g(z,\mathtt{c})$ and we get the equality $z^{i-1}x=g(z,\mathtt{c})/z$, where 
the right hand side is in $\mathbb{C}[z,\mathtt{c}]$. Repeating this argument finitely many times
we get $x\in \mathbb{C}[z,\mathtt{c}]$, and, consequently, $\varphi(x)\in \mathbb{C}[z,z(h+\frac{3}{2})^2]$.
\end{proof}

\subsection{Center of $U$}\label{s5.3}

From Proposition~\ref{prop22} we get the following description of $Z(\mathfrak{s})$ 
which corrects \cite[Theorem~1.1(1)]{WZ1}.

\begin{corollary}\label{prop23}
We have $Z(\mathfrak{s})=\mathbb{C}[z,\mathtt{c}]$.
\end{corollary}

\begin{proof}
This follows from Proposition~\ref{prop22} and the observation that we have $\varphi(z)=z$ and
$\varphi(\mathtt{c})=z(h+\frac{3}{2})^2-\frac{1}{4}z$. 
\end{proof}

\subsection{$U$ is free over the center}\label{s5.4}

\begin{corollary}\label{prop24}
The algebra  $U$ is free as a $Z(\mathfrak{s})$-module.
\end{corollary}

\begin{proof}
The algebra $U$ has the usual filtration by degree of monomials, let 
$\overline{U}$ be the associated graded algebra. The image of the sequence $(z,\mathtt{c})$ is a regular
sequence in $\overline{U}$ (which means that $z$ is neither a zero divisor nor invertible in
$\overline{U}$ and the image of $\mathtt{c}$ in $\overline{U}/(z)$ is again neither a zero divisor 
nor invertible). Now the claim of our corollary follows from \cite[Theorem~1.1]{FO}.
\end{proof}

For $\dot{z}\in\mathbb{C}$ consider the algebra $U_{\dot{z}}:=U/U(z-\dot{z})$. For simplicity we
denote elements in $U$ and their images in $U_{\dot{z}}$ by the same symbol.

\begin{proposition}\label{prop25}
\begin{enumerate}[$($i$)$]
\item\label{prop25.1} $U_{\dot{z}}$ is a free $\mathbb{C}[\mathtt{c}]$-module.
\item\label{prop25.2} For any maximal ideal $\mathfrak{m}$ in $\mathbb{C}[\mathtt{c}]$,
the left multiplication action of $t:=2{\dot{z}}f+q^2$ on $U_{\dot{z}}/U_{\dot{z}}\mathfrak{m}$ is injective.
\end{enumerate}
\end{proposition}

\begin{proof}
Claim~\eqref{prop25.1} follows immediately from Corollary~\ref{prop24}. To prove claim~\eqref{prop25.2}
we first consider the case $\dot{z}\neq 0$. Let $H$ denote the subspace of $U(\mathfrak{sl}_2)$ which
is a linear combination of monomials of the form $f^ih^j$ and $h^ie^j$. Then $H$ contains a basis of
$U(\mathfrak{sl}_2)$ as (both left and right) $\mathbb{C}[\underline{\mathtt{c}}]$-module,
see e.g. \cite[Theorem~2.33]{Ma}. From Subsection~\ref{s2.2} we have that
$\mathtt{c}=\dot{z}\underline{\mathtt{c}}+u$ where $u$ is a linear combination of monomials which
never contain both factors $e$ and $f$ at the same time. It follows that any basis in 
$H\otimes \mathbb{C}[p,q]$ is a basis of $U_{\dot{z}}$ over $\mathbb{C}[\mathtt{c}]$. Consider 
the standard monomial basis in $H\otimes \mathbb{C}[p,q]$ as follows:
\begin{displaymath}
\{f^ah^be^cq^dp^s\,\vert\, a,b,c,d,s\in\mathbb{Z}_+, ac=0\}. 
\end{displaymath}
Introduce the following linear order $\preceq$ on elements of this basis:
Set $f^ah^be^cq^dp^s\prec f^{a'}h^{b'}e^{c'}q^{d'}p^{s'}$ if:
\begin{itemize}
\item $a+b+c< a'+b'+c'$;
\item $a+b+c=a'+b'+c'$ but $a<a'$;
\item $a+b+c=a'+b'+c'$  and $a=a'$ but $b<b'$;
\item $a+b+c=a'+b'+c'$  and $a=a'$ and $b=b'$ but $d<d'$;
\item $a+b+c=a'+b'+c'$  and $a=a'$ and $b=b'$ and $d=d'$ but $s<s'$.
\end{itemize}
For any $u\in\mathbb{C}[\mathtt{c}]\setminus 0$ we have
\begin{displaymath}
t\cdot f^ah^be^cq^dp^s u=
\begin{cases}
f^{a+1}h^be^cq^dp^s u'  + \text{smaller terms}, & c=0;\\ 
h^{b+2}e^{c-1}q^dp^s u'  + \text{smaller terms}, & c>0; 
\end{cases}
\end{displaymath}
where $u'$ is obtained from $u$ by multiplying with a nonzero constant and
``smaller terms'' means a linear combination of monomials (with coefficients from $\mathbb{C}[\mathtt{c}]$)
which are smaller with respect to $\prec$.
From this it follows that if $x$ and $y$ are two monomials such that $x\prec y$,
$u_1,u_2\in\mathbb{C}[\mathtt{c}]\setminus 0$ and 
$x'$ and $y'$ are highest monomials (with respect to $\prec$)
which appear with nonzero coefficients in $t\cdot xu_1$ and
$t\cdot yu_2$, respectively, then $x'\prec y'$. 

Let $\omega$ be a nonzero element of $U_{\dot{z}}/U_{\dot{z}}\mathfrak{m}$.
Write $\omega=xu+ \text{smaller terms}$, where $x$ is the maximal monomial with respect to 
$\prec$ which appears in $\omega$ and $u\in \mathbb{C}[\mathtt{c}]\setminus\mathfrak{m}$.
Let $y$ be the maximal monomial which appears in $t\cdot xu$. Then the previous paragraph
implies that $y$ appears in $t\cdot \omega$ with coefficient $c\cdot u$ for some nonzero constant $c$.
Hence $t\cdot \omega\neq 0$ and we are done.

It remains to consider the case $\dot{z}=0$. In this case we will prove that the left multiplication
with $q$ on $U_{\dot{z}}/U_{\dot{z}}\mathfrak{m}$ is injective. Using the PBW theorem,
we choose the following basis of $U_{\dot{z}}$ over $\mathbb{C}[\mathtt{c}]$:
\begin{displaymath}
\{q^ap^bh^cf^de^s\,\vert\, a,b,c,d,s\in\mathbb{Z}_+, abc=0\}. 
\end{displaymath}
Similarly to the above, introduce the linear ordering $\prec$ on monomials as follows:
Set $q^ap^bh^cf^de^s\prec q^{a'}p^{b'}h^{c'}f^{d'}e^{s'}$ if:
\begin{itemize}
\item $a<a'$;
\item $a=a'$ and $\min\{b,c\}<\min\{b',c'\}$;
\item $a=a'$ and $\min\{b,c\}=\min\{b',c'\}$ but $b<b'$;
\item $a=a'$ and $\min\{b,c\}=\min\{b',c'\}$ and $b=b'$ but $c<c'$;
\item $a=a'$ and $\min\{b,c\}=\min\{b',c'\}$ and $b=b'$ and $c=c'$ but $d<d'$;
\item $a=a'$ and $\min\{b,c\}=\min\{b',c'\}$ and $b=b'$ and $c=c'$ and $d=d'$ but $s<s'$.
\end{itemize}
Set $\tau:=\min\{b,c\}$.
Then for any $u\in\mathbb{C}[\mathtt{c}]\setminus 0$ we have
\begin{displaymath}
q\cdot q^ap^bh^cf^de^s u=
\begin{cases}
q^{a+1}p^bh^cf^de^s u'  + \text{smaller terms}, & bc=0;\\ 
q^{\tau+1}p^{b-\tau}h^{c-\tau}f^{d}e^{s+\tau} u'  + \text{smaller terms}, & bc>0; 
\end{cases}
\end{displaymath}
where $u'$ is obtained from $u$ by multiplying with a nonzero constant and
``smaller terms'' means a linear combination of monomials (with coefficients from $\mathbb{C}[\mathtt{c}]$)
which are smaller with respect to $\prec$.
Now the proof is completed by the same arguments as in the case $\dot{z}\neq 0$.
\end{proof}

\subsection{Annihilators of Verma modules}\label{s5.5}

Our aim in this subsection is to prove the following statement which corrects \cite[Theorem~1.1(2)]{WZ1}.

\begin{theorem}\label{thm401}
The annihilator in $U$ of $\Delta(\lambda)$ is centrally generated, that is,
$\mathrm{Ann}_{U}\Delta(\lambda)=U\mathrm{Ann}_{Z(\mathfrak{s})}\Delta(\lambda)$.
\end{theorem}

For $\lambda\in\mathfrak{h}^*$ let $\mathbf{m}_{\lambda}$ be the maximal ideal in $Z(\mathfrak{s})$
such that $\mathbf{m}_{\lambda}\Delta(\lambda)=0$. The ideal $\mathbf{m}_{\lambda}$ is generated by
$z-\lambda(z)$ and $\mathtt{c}-\vartheta_{\lambda}$. The assertion of Theorem~\ref{thm401} can be
reformulated as follows: the annihilator in $U$ of $\Delta(\lambda)$ is the ideal $U\mathbf{m}_{\lambda}$.

\begin{proof}
Clearly,  $U\mathbf{m}_{\lambda}$ annihilates $\Delta(\lambda)$, so we only need to prove the opposite inclusion.
Set $\dot{z}:=\lambda(z)$ and consider the quotient algebras $U_{\dot{z}}:=U/U(z-\dot{z})$
and $\tilde{U}_{\dot{z}}:=U/\mathrm{Ann}_U(\Delta(\lambda))$. Clearly,
$U_{\dot{z}}$ is a domain and $U_{\dot{z}}\tto \tilde{U}_{\dot{z}}$. For simplicity we will use the 
same notation for elements in $U$ and their images in both $U_{\dot{z}}$ and $\tilde{U}_{\dot{z}}$. 
The module $\Delta(\lambda)$ is naturally both a $U_{\dot{z}}$-module and a $\tilde{U}_{\dot{z}}$-module. 
 
Consider the multiplicative set $\{t^i\,\vert\,i\in\mathbb{Z}_+\}$, where 
\begin{displaymath}
t:=\begin{cases}
2{\dot{z}}f+q^2, &  \dot{z}\neq 0;\\
q, &  \dot{z}=0.
\end{cases}
\end{displaymath}
As the adjoint action of $t$ on $U_{\dot{z}}$ is locally nilpotent, 
$\{t^i\,\vert\,i\in\mathbb{Z}_+\}$ is an Ore set by \cite[Lemma~4.2]{Mt}.
Therefore we can consider the corresponding Ore localization $U'_{\dot{z}}$ of $U_{\dot{z}}$ 
and also the Ore localization $\tilde{U}'_{\dot{z}}$ of $\tilde{U}_{\dot{z}}$.
The element $t$ obviously acts injectively on $\Delta(\lambda)$ and hence $\Delta(\lambda)$
embeds (as a $U_{\dot{z}}$-submodule) into the localized modules
$U'_{\dot{z}}\otimes_{U_{\dot{z}}}\Delta(\lambda)$ and
$\tilde{U}'_{\dot{z}}\otimes_{\tilde{U}_{\dot{z}}}\Delta(\lambda)$.


Let $\mathfrak{a}$ denote the Lie subalgebra of $\mathfrak{s}$ spanned by $f,h,p,q,z$ and set
$A:=U(\mathfrak{a})/U(\mathfrak{a})(z-\dot{z})$ which is naturally a subalgebra of $U_{\dot{z}}$.

\begin{lemma}\label{lem402}
We have $A\cap \mathrm{Ann}_{U_{\dot{z}}}(\Delta(\lambda))=0$.
\end{lemma}

\begin{proof}
The set $I:=A\cap \mathrm{Ann}_{U_{\dot{z}}}(\Delta(\lambda))$ is a two-sided ideal in $A$.
Assume $u$ is a nonzero element of $I$. Write 
$u=\sum_{k\geq 0}\beta_k(h,f,q)p^k$ for some $\beta_k(h,f,q)\in U(\tilde{\mathfrak{n}}_-)$,
where $\tilde{\mathfrak{n}}_-$ is the Lie algebra spanned by $f,q$ and $h$.

Consider first the case $\dot{z}=0$. We prove, by induction on
$k$, that $\beta_k(h,f,q)=0$ for all $k$. For $m\geq 0$ let $M_m$ denote the linear subspace of 
$\Delta(\lambda)$ generated 
by all elements of the form $f^mq^iv_{\lambda}$, $i\in\mathbb{Z}_+$. As $U(\mathfrak{n}_-)$ acts freely on
$\Delta(\lambda)$, we have that $f^mq^iv_{\lambda}$, $i\in\mathbb{Z}_+$, is, in fact, a basis in $M_m$.
Note that all elements in this basis have different $h$-weights.
As $\dot{z}=0$, we have
\begin{equation}\label{eq983}
u\cdot M_m= \sum_{k=0}^m\beta_k(h,f,q)p^k\cdot M_m=0.
\end{equation}
In particular, $u\cdot  M_0=\beta_0(h,f,q)\cdot M_0=0$. As $M_0$ contains nonzero elements of infinitely many 
$h$-weights and $U(\mathfrak{n}_-)$ acts freely on $\Delta(\lambda)$, it follows that $\beta_0(h,f,q)=0$.
Indeed, write $\beta_0(h,f,q)=\sum_{i,j}f^iq^j\gamma_{i,j}(h)$ for some $\gamma_{i,j}(h)\in\mathbb{C}[h]$.
Only finitely many of the $\gamma_{i,j}$'s are nonzero. Find $0\neq v\in M_0$ such that $\gamma_{i,j}\cdot v=
c_{i,j}v$ for some nonzero $c_{i,j}\in\mathbb{C}$ whenever $\gamma_{i,j}\neq 0$ (this is possible since $M_0$ 
contains nonzero elements of infinitely many $h$-weights). Then $\beta_0(h,f,q)\cdot v=
\sum_{i,j}c_{i,j}f^iq^j\cdot v$. Since the action of the domain $U(\mathfrak{n}_-)$ on $\Delta(\lambda)$ is free, 
$\sum_{i,j}c_{i,j}f^iq^j\cdot v$ is nonzero as soon as $\sum_{i,j}c_{i,j}f^iq^j$ is.
This contradicts $\beta_0(h,f,q)\cdot M_0=0$ and hence implies $\beta_0(h,f,q)=0$.

Assume now that we have $\beta_i(h,f,q)=0$ for all $i<k$. Then we have $u\cdot  M_k=\beta_k(h,f,q)\cdot M_k=0$ 
by \eqref{eq983}. Similarly to the above, since $M_k$ contains nonzero elements of infinitely 
many $h$-weights and $U(\mathfrak{n}_-)$ acts freely on $\Delta(\lambda)$, it follows that $\beta_k(h,f,q)=0$.
Hence $u=0$, a contradiction.

The case $\dot{z}\neq 0$ is proved by replacing $q$ with $2\dot{z} f+q^2$ (the latter element commutes with $p$), 
and $f$ with $q$  in the definition of $M_m$ and following the proof for the case $\dot{z}=0$.
\end{proof}

Let $J$ denote the ideal of $U_{\dot{z}}$ generated by $\mathtt{c}-\vartheta_{\lambda}$ and
$J'$ denote the ideal of $U'_{\dot{z}}$ generated by $\mathtt{c}-\vartheta_{\lambda}$. Note that
in $U'_{\dot{z}}$ the relation $\mathtt{c}-\vartheta_{\lambda}=0$ can be equivalently written as
$e=y$ where $y$ is in the subalgebra $A'$ of $U'_{\dot{z}}$ generated by $A$ and $t^{-1}$ 
(here our special choice of $t$ is crucial). Clearly, $A'$ is the localization of $A$ at $t$.

Similarly to \cite[Theorem~3.32]{Ma} one shows that $U'_{\dot{z}}$ has a PBW basis consisting of
all monomials of the form $t^iq^lh^jp^ke^m$ (here $i\in\mathbb{Z}$ and $l,j,k,m\in\mathbb{Z}_+$)
if $\dot{z}\neq 0$. If $\dot{z}= 0$ the basis consists of the monomials $t^if^lh^jp^ke^m$
(here $i\in\mathbb{Z}$ and $l,j,k,m\in\mathbb{Z}_+$). 
From the previous paragraph it follows that  $U'_{\dot{z}}/J'$ has a PBW basis consisting of
all monomials of the form $t^iq^lh^jp^k$ if $\dot{z}\neq 0$, respectively, of the form
$t^if^lh^jp^k$ if $\dot{z}= 0$. Let us collect what we now know in the diagram:
\begin{displaymath}
\xymatrix{ 
&&\tilde{U}_{\dot{z}}\ar@{^(->}[rr]&&\tilde{U}'_{\dot{z}}&&\\
&&&J'\ar@{^(->}[rd]&&&\\
U\ar@{->>}[rr]&&U_{\dot{z}}\ar@{^(->}[rr]\ar@{->>}[uu]&&U'_{\dot{z}}\ar@{->>}[rr]\ar@{-->>}[uu]
&&U'_{\dot{z}}/J'\ar@{=>>}[lluu]\\
&&&&&&\\
&&A\ar@{^(->}[rr]\ar@{^(->}[uu]&&A'\ar@{^(~>}[uu]\ar@{==>}[uurr]^{\sim}\ar@/^2pc/@{^(.>}[uuuu]&&
}
\end{displaymath}
here all solid maps are natural inclusions or projections. The dashed arrow from $U'_{\dot{z}}$
to $\tilde{U}'_{\dot{z}}$ comes from the universal property of localization. Similarly the tilted map
from $A'$ to $U'_{\dot{z}}$. Both these maps make the corresponding squares commutative. 
Since the dashed map sends $J'$ to zero, it factors as the double solid map via $U'_{\dot{z}}/J'$.
From Lemma~\ref{lem402} it follows that both maps from $A'$, namely the tilded map to 
$U'_{\dot{z}}$ and the dotted to $\tilde{U}'_{\dot{z}}$ are injective. From the previous
paragraph we get that the double dashed composition map from $A'$ to $U'_{\dot{z}}/J'$
is an isomorphism. The diagram clearly commutes. From the commutativity  it follows
that the dotted map is an isomorphism and hence $U'_{\dot{z}}/J'\cong \tilde{U}'_{\dot{z}}$.

Now assume that $u\in U_{\dot{z}}$ annihilates $\Delta(\lambda)$. Then the previous paragraph implies
that $u$ annihilates $\tilde{U}'_{\dot{z}}\otimes_{\tilde{U}_{\dot{z}}}\Delta(\lambda)\cong
U'_{\dot{z}}\otimes_{U_{\dot{z}}}\Delta(\lambda)$ and therefore belongs to
$J'$. This means that $t^{i}u\in U\mathbf{m}_{\lambda}$ for some $i\in\mathbb{Z}_+$. 
From Proposition~\ref{prop25}\eqref{prop25.2} it now follows that
$u\in U\mathbf{m}_{\lambda}$,  completing the proof.
\end{proof}

As a corollary from Theorem~\ref{thm401} and Corollary~\ref{prop23} we obtain:

\begin{corollary}\label{prop2353}
The element $\kappa:=fp^2-eq^2-hpq$ generates $Z(\overline{\mathfrak{s}})$.
\end{corollary}

\begin{proof}
That $\kappa\in Z(\overline{\mathfrak{s}})$ follows directly from Corollary~\ref{prop23} by factoring $z$ out. 
Conversely, assume that $Z(\overline{\mathfrak{s}})\neq \mathbb{C}[\kappa]$ and let
$a\in Z(\overline{\mathfrak{s}})\setminus\mathbb{C}[\kappa]$ be an element of minimal total monomial degree. 
Consider a Verma $\mathfrak{s}$-module $\Delta(\lambda)$ with zero central charge.
Then $\Delta(\lambda)$ has the structure of a Verma $\overline{\mathfrak{s}}$-module by restriction. 
The element $a$ thus acts as a scalar on $\Delta(\lambda)$ and hence $a-\dot{a}$ annihilates
$\Delta(\lambda)$ for some $\dot{a}\in\mathbb{C}$. Therefore, by Theorem~\ref{thm401}, we can write 
$a-\dot{a}=u\kappa$ for some $u\in U(\overline{\mathfrak{s}})$. As $U(\overline{\mathfrak{s}})$ is a
domain, we get $u\in Z(\overline{\mathfrak{s}})$. Moreover, $u$ has strictly smaller degree than $a$.
Therefore $u\in \mathbb{C}[\kappa]$ and hence $a\in \mathbb{C}[\kappa]$, a contradiction. The claim follows.
\end{proof}

\section{Harish-Chandra bimodules and primitive ideals}\label{s6}

\subsection{Locally finite dimensional $\mathfrak{s}$-modules}\label{s6.1}

Denote by $U\text{-}\mathrm{fdmod}$ the full subcategory of $U\text{-}\mathrm{mod}$ consisting of
all finite dimensional modules. Clearly, $\mathcal{O}^{f}$ is a subcategory in $U\text{-}\mathrm{fdmod}$,
however, there exist objects in $U\text{-}\mathrm{fdmod}$ which are not isomorphic to any object in
$\mathcal{O}^{f}$. Indeed, by definition $z$ annihilates all objects in $\mathcal{O}^{f}$. On the other hand,
by \cite[Theorem~2.5.7]{Di} the intersection of annihilators in $U$ of all objects in $U\text{-}\mathrm{fdmod}$ is zero.

Denote by $U\text{-}\mathrm{lfdMod}$ the full subcategory of $U\text{-}\mathrm{Mod}$ consisting of all
{\em locally finite dimensional modules}, that is all $M\in U\text{-}\mathrm{Mod}$ such that 
$\dim U v<\infty$ for any $v\in M$. Clearly, $U\text{-}\mathrm{fdmod}$ is a subcategory of
$U\text{-}\mathrm{lfdMod}$. 
Denote by $U\text{-}\mathrm{zlm}$ the full subcategory of $U\text{-}\mathrm{lfdMod}$ consisting of 
all modules annihilated by $z$. Clearly, $\mathcal{O}^{f}$ is a subcategory of
$U\text{-}\mathrm{zlm}$ and $U\text{-}\mathrm{zlm}$ itself is a subcategory of $U\text{-}\mathrm{lfdMod}$. 
Both $U\text{-}\mathrm{lfdMod}$ and $U\text{-}\mathrm{zlm}$ are locally
noetherian Grothendieck categories (see \cite[Appendix~A]{Kr} or \cite{Ro}). In particular, each injective
object in these categories is a coproduct of indecomposable injective objects and this decomposition is
unique up to isomorphism. The standard universal coextension procedure using simple finite dimensional 
modules gives that each object in  both $U\text{-}\mathrm{lfdMod}$ and $U\text{-}\mathrm{zlm}$ is a subobject 
of an injective object. Simple objects in both $U\text{-}\mathrm{lfdMod}$ and $U\text{-}\mathrm{zlm}$ are simple
finite dimensional $\mathfrak{sl}_2$-modules.

An object $M\in U\text{-}\mathrm{lfdMod}$ is said to be of {\em finite type} provided that 
$\dim\mathrm{Hom}_U(V,M)<\infty$ for any simple finite dimensional $V$.

For $n\in\mathbb{Z}_+$ denote by $I^f(n)$ the injective envelope in $U\text{-}\mathrm{zlm}$ of the 
simple $n+1$-dimensional $U$-module.

\begin{lemma}\label{lem753}
We have $I^f(n)\cong I^f(0)\otimes \mathrm{soc}(I^f(n))$.
\end{lemma}

\begin{proof}
Let $V=\mathrm{soc}(I^f(n))$. As $\mathrm{F}_V$ is biadjoint to itself, it maps injective modules to
injective modules. For a simple finite dimensional $U$-module $V'$ we have
\begin{displaymath}
\mathrm{Hom}_{U}(V',I^f(0)\otimes V)\cong  
\mathrm{Hom}_{U}(V'\otimes V^*,I^f(0)). 
\end{displaymath}
As $V'\otimes V^*$ has a trivial submodule if and only if $V'\cong V$ (see \cite[Theorem~1.39]{Ma}), 
the claim follows.
\end{proof}

\subsection{Harish-Chandra bimodules}\label{s6.2}

For a $U\text{-}U$-bimodule $X$ we denote by $X^{\mathrm{ad}}$ the adjoint $\mathfrak{s}$-module
(that is the $\mathfrak{s}$-module on the underlying vector space $X$ where the action of $a\in \mathfrak{s}$
is given by $a\cdot x=ax-xa$). A finitely generated $U\text{-}U$-bimodule $X$ is called a {\em weak Harish-Chandra} 
bimodule provided that $X^{\mathrm{ad}}\in U\text{-}\mathrm{lfdMod}$ and is of finite type. 
A finitely generated $U\text{-}U$-bimodule $X$ is called a {\em Harish-Chandra} 
bimodule provided that it is a weak Harish-Chandra bimodule and $X^{\mathrm{ad}}\in U\text{-}\mathrm{zlm}$.
We denote by $\widetilde{\mathcal{H}}$ the category of all weak Harish-Chandra bimodules for $U$.
We denote by ${\mathcal{H}}$ the category of all Harish-Chandra bimodules for $U$.

For $M,N\in U\text{-}\mathrm{Mod}$ the vector space $\mathrm{Hom}_{\mathbb{C}}(M,N)$ carries the natural structure
of a $U\text{-}U$-bimodule (coming from the $U$-module structures on $M$ and $N$). 
Denote by $\widetilde{\mathcal{L}}(M,N)$ the subspace of $\mathrm{Hom}_{\mathbb{C}}(M,N)$
consisting of all elements, the adjoint action of $\mathfrak{s}$ on which is locally finite. As usual,
see \cite[1.7.9]{Di}, the space  $\widetilde{\mathcal{L}}(M,N)$ is, in fact, 
a subbimodule of $\mathrm{Hom}_{\mathbb{C}}(M,N)$.
Denote by ${\mathcal{L}}(M,N)$ the subbimodule of $\widetilde{\mathcal{L}}(M,N)$ consisting of all
elements annihilated by the adjoint action of $z$.
For a finite dimensional $\mathfrak{s}$-module $V$ we have the following isomorphism (see \cite[6.8]{Ja}):
\begin{equation}\label{eq500}
\mathrm{Hom}_{U}(V,\widetilde{\mathcal{L}}(M,N)^{\mathrm{ad}})\cong
\mathrm{Hom}_{U}(V\otimes M,N)\cong
\mathrm{Hom}_{U}(M,V^*\otimes N).
\end{equation}

\begin{lemma}\label{lem501}
If $M,N\in\mathcal{O}$, then $\widetilde{\mathcal{L}}(M,N)= {\mathcal{L}}(M,N)$ and
the latter is a Harish-Chandra bimodule for $U$.
\end{lemma}

\begin{proof}
Each object in  $\mathcal{O}$ is finitely generated and hence decomposes into a finite direct sum of 
indecomposable objects. By additivity, it is enough to prove the claim for indecomposable $M$ and $N$.
Assume $M$ and $N$ are indecomposable. Since $z$ annihilates each simple finite dimensional $\mathfrak{s}$-module, 
for $\widetilde{\mathcal{L}}(M,N)$  to be nonzero $z$ should act with the same scalar on $M$ and $N$, in particular,
it follows that $z$ annihilates $\mathrm{Hom}_{\mathbb{C}}(M,N)$ and thus $\widetilde{\mathcal{L}}(M,N)$.
This implies $\widetilde{\mathcal{L}}(M,N)= {\mathcal{L}}(M,N)$. 

The claim that ${\mathcal{L}}(M,N)$ is a Harish-Chandra bimodule follows from \eqref{eq500} and the observation
that all homomorphism spaces in $\mathcal{O}$ are finite dimensional.
\end{proof}

For $M\in\mathcal{O}$ we thus get a canonical inclusion of $U\text{-}U$-bimodules.
\begin{equation}\label{eq325}
U/\mathrm{Ann}_U(M)\hookrightarrow  {\mathcal{L}}(M,M).
\end{equation}

\begin{lemma}\label{lem329}
Let $M$ be projective in $\mathcal{O}$. Then ${\mathcal{L}}(M,M)^{\mathrm{ad}}$ is injective in 
$U\text{-}\mathrm{zlm}$.
\end{lemma}

\begin{proof}
The functor $V\mapsto \mathrm{Hom}_{U}(V,\widetilde{\mathcal{L}}(M,M)^{\mathrm{ad}})$ is
exact by \eqref{eq500}, projectivity of $M$, exactness of $*$ and exactness of tensoring over a field.
The claim follows from this observation and Lemma~\ref{lem501}.
\end{proof}

\begin{corollary}\label{cor330}
Let $\lambda\in\mathfrak{h}^*$ be such that $\lambda(z)\neq 0$. Then
${\mathcal{L}}(\Delta(\lambda),\Delta(\lambda))^{\mathrm{ad}}$ is injective in 
$U\text{-}\mathrm{zlm}$.
\end{corollary}

\begin{proof}
If $\Delta(\lambda)$ is projective, the claim follows from Lemma~\ref{lem329}. If $\Delta(\lambda)$ is not
projective, then we are in the situation described in Proposition~\ref{prop4}. In particular, we have
a short exact sequence
\begin{equation}\label{eq334}
0\to \Delta(\lambda)\to \Delta(r\cdot \lambda)\to L(r\cdot \lambda)\to 0.
\end{equation}
Using \eqref{eq500}, the fact that Gelfand-Kirillov dimension of $L(r\cdot \lambda)$ is strictly smaller than that
of $\Delta(\lambda)$ and the fact that projective functors do not affect Gelfand-Kirillov dimension, 
we get 
\begin{equation}\label{eq335}
\mathcal{L}(\Delta(\lambda),L(r\cdot \lambda))=\mathcal{L}(L(r\cdot \lambda),\Delta(r\cdot \lambda))=0.
\end{equation}
Applying the left exact functor $\mathcal{L}(\Delta(\lambda),{}_-)$ to \eqref{eq334} and using \eqref{eq335} we get
\begin{displaymath}
\mathcal{L}(\Delta(\lambda),\Delta(\lambda)) \cong
\mathcal{L}(\Delta(\lambda),\Delta(r\cdot \lambda)). 
\end{displaymath}
Applying the left exact functor $\mathcal{L}({}_-,\Delta(r\cdot \lambda))$ to \eqref{eq334} and using 
\eqref{eq335} we thus get a natural inclusion
\begin{displaymath}
\mathcal{L}(\Delta(r\cdot \lambda),\Delta(r\cdot \lambda)) \subset 
\mathcal{L}(\Delta(\lambda),\Delta(r\cdot \lambda))\cong \mathcal{L}(\Delta(\lambda),\Delta(\lambda)). 
\end{displaymath}
As $\Delta(r\cdot \lambda)$ is projective, $\mathcal{L}(\Delta(r\cdot \lambda),\Delta(r\cdot \lambda))^{\mathrm{ad}}$
is injective by Lemma~\ref{lem329} and hence splits as a direct summand inside 
$\mathcal{L}(\Delta(\lambda),\Delta(\lambda))^{\mathrm{ad}}$. To complete the proof it is therefore enough to use  
\eqref{eq500} and check that 
\begin{displaymath}
\dim \mathrm{Hom}_{U}(\Delta(\lambda),V\otimes \Delta(\lambda)) =
\dim \mathrm{Hom}_{U}(\Delta(r\cdot \lambda),V\otimes \Delta(r\cdot \lambda))
\end{displaymath}
for any simple finite dimensional $\mathfrak{sl}_2$-module $V$. This is a straightforward computation using 
Proposition~\ref{prop4}\eqref{prop4.5}.
\end{proof}

\subsection{The bimodules ${\mathcal{L}}(\Delta(\lambda),\Delta(\lambda))$ for nonzero central charge}\label{s6.3}

\begin{proposition}\label{prop327}
Let $\lambda\in\mathfrak{h}^*$ be such that $\lambda(z)\neq 0$. Then 
the canonical inclusion \eqref{eq325} for $M=\Delta(\lambda)$ is an isomorphism.
\end{proposition}

\begin{proof}
We only have to prove surjectivity. Let $V$ be a simple finite dimensional $\mathfrak{sl}_2$-module
of dimension $n$. Then $V\otimes \Delta(\lambda)$ has a Verma filtration with subquotients
\begin{displaymath}
\Delta(\lambda+(n-1)h^{\checkmark}),\Delta(\lambda+(n-3)h^{\checkmark}),
\Delta(\lambda+(n-5)h^{\checkmark}),\dots, \Delta(\lambda-(n-1)h^{\checkmark}),
\end{displaymath}
each occurring with multiplicity one. From our explicit description of blocks with nonzero central charge
in Section~\ref{s3} it follows that if $n$ is even, then there are no homomorphisms from $\Delta(\lambda)$
to any of these subquotients. Hence 
$\mathrm{Hom}_{U}(V,{\mathcal{L}}(M,M)^{\mathrm{ad}})=0$ by \eqref{eq500}.

If $n$ is odd, we have two possibilities. The first one is that $\Delta(\lambda)$ is the only Verma module
from the block which appears as a subquotient in the above list. In this case we obviously get 
$\mathrm{Hom}_{U}(V,{\mathcal{L}}(M,M)^{\mathrm{ad}})=1$ by \eqref{eq500}.
The second case is that the other Verma module from the same block as $\Delta(\lambda)$ also appears in the
above list. In this case one checks that the projection of $V\otimes \Delta(\lambda)$ is the indecomposable
projective cover of a simple Verma module in the block (cf \cite[Chapter~5]{Ma}) and hence again 
$\mathrm{Hom}_{U}(V,{\mathcal{L}}(M,M)^{\mathrm{ad}})=1$ by \eqref{eq500}. Altogether we get
\begin{equation}\label{eq505}
\dim\mathrm{Hom}_{U}(V,{\mathcal{L}}(M,M)^{\mathrm{ad}})=\dim V_0.
\end{equation}
This and Corollary~\ref{cor330} together imply that ${\mathcal{L}}(M,M)^{\mathrm{ad}}$ is a multiplicity free 
direct sum of injective envelopes (in $U\text{-}\mathrm{zlm}$) of all odd-dimensional simple 
$U$-modules.

Now let us estimate $U/\mathrm{Ann}_U(M)$. We know that $\mathrm{Ann}_U(M)=U\mathbf{m}_{\lambda}$
by Theorem~\ref{thm401}. The algebra $U(\mathfrak{i})$ acts on $M$ via the simple quotient
$U(\mathfrak{i})/(z-\lambda(z))$ which is isomorphic to the first Weyl algebra.
It is straightforward to check, using computation and results from Subsection~\ref{s4.4}, that
$U(\mathfrak{i})/(z-\lambda(z))^{\mathrm{ad}}$ is isomorphic to the injective hull of the trivial
module. Since $\lambda(z)\neq 0$, we have  $\mathtt{c}=\lambda(z)\underline{\mathtt{c}}+x$
where $x$ of lower $U(\mathfrak{sl}_2)$-degree. Therefore we may use the PBW theorem to produce
a vector space decomposition
\begin{displaymath}
U/\mathrm{Ann}_U(M)\cong
U(\mathfrak{i})/(z-\lambda(z))\otimes U(\mathfrak{sl}_2)/(\mathtt{c})
\end{displaymath}
compatible with the adjoint action. The adjoint module $\left[U(\mathfrak{sl}_2)/(\mathtt{c})\right]^{\mathrm{ad}}$ 
is a multiplicity free direct sum of all simple odd-dimensional modules. From Lemma~\ref{lem753}
we thus get that $U/\mathrm{Ann}_U(M)$ is a multiplicity free 
direct sum of injective envelopes (in $U\text{-}\mathrm{zlm}$) of all odd-dimensional simple 
$U$-modules. The claim follows.
\end{proof}

\subsection{Primitive ideals for nonzero central charge}\label{s6.4}

The following statement describes all primitive ideals for $U$ with nonzero central charge.

\begin{theorem}\label{thm777}
Let $\lambda\in\mathfrak{h}^*$ be such that $\lambda(z)\neq 0$ and set $U_{\lambda}:=U/U\mathbf{m}_{\lambda}$.
\begin{enumerate}[$($i$)$]
\item\label{thm777.1} If $\lambda(h)\not\in\frac{1}{2}+\mathbb{Z}$ or $\lambda(h)=-\frac{3}{2}$, 
then $U_{\lambda}$ is a simple algebra.
\item\label{thm777.2} If $\lambda(h)\in\{-\frac{1}{2},\frac{1}{2},\frac{3}{2},\frac{5}{2},\dots \}$,
then $U_{\lambda}$ has two primitive ideals, namely $0$ and $\mathrm{Ann}_{U_{\lambda}}(L(\lambda))$.
\end{enumerate}
\end{theorem}

We note that for $\lambda\in \{-\frac{5}{2},-\frac{7}{2},-\frac{9}{2},\dots \}$ we have 
$U_{\lambda}=U_{r\cdot \lambda}$ and hence this case reduces to Theorem~\ref{thm777}\eqref{thm777.2}.

\begin{proof}
Set $\dot{z}:=\lambda(z)\neq 0$ and consider the associative algebra $B_{\dot{z}}=U(\mathfrak{i})/(z-\dot{z})$
as in Subsection~\ref{s3.9} (which is isomorphic to the first Weyl algebra, in particular, it is a simple algebra).
Consider the simple $B_{\dot{z}}$-module $\mathbf{M}:=B_{\dot{z}}/B_{\dot{z}}p$ which, following
Subsection~\ref{s3.9}, can be regarded as the simple highest weight module
$L(\mu)$ where $\mu(z)=\dot{z}$ and $\mu(h)=-\frac{1}{2}$. From \cite[Theorem~1]{LMZ1} it thus follows that 
$U/\mathrm{Ann}_U(L(\mu))=B_{\dot{z}}$ is a simple algebra.

If $M$ is a simple $\mathfrak{sl}_2$-module, then $M\otimes \mathbf{M}$ is a simple $U$-module by \cite[Theorem~3]{LMZ1}.
If $M$ is simple finite dimensional but not one-dimensional, then clearly 
$M\otimes \mathbf{M}$ has a different highest weight
than $\mathbf{M}$ and thus $M\otimes \mathbf{M}\not\cong \mathbf{M}$. 
Hence \eqref{eq500} implies that we have the inclusion
$\mathcal{L}(L(\mu),L(\mu))^{\mathrm{ad}}\subset I^f(0)$. From Theorem~\ref{thm301}
we get that $I^f(0)$ is a uniserial module. This implies that each proper submodule of $I^f(0)$ is finite dimensional.
We know that $\mathcal{L}(L(\mu),L(\mu))^{\mathrm{ad}}$ contains $U/\mathrm{Ann}_U(L(\mu))$ which is 
infinite dimensional. This means that $\mathcal{L}(L(\mu),L(\mu))=U/\mathrm{Ann}_U(L(\mu))$. 
Since the latter is a simple algebra, the $U\text{-}U$-bimodule $\mathcal{L}(L(\mu),L(\mu))$ is simple.

If $M$ is a Verma $\mathfrak{sl}_2$-module with a non-integral highest weight, then $M$ is simple and
$M\otimes \mathbf{M}$ is a simple 
Verma $U$-module, say $\Delta(\nu)$, moreover,  $\nu(h)\not\in\frac{1}{2}+\mathbb{Z}$.
By \cite[7.25]{Ja} we have $\mathcal{L}(M,M)\cong U(\mathfrak{sl}_2)/\mathrm{Ann}_{\mathfrak{sl}_2}(M)$
and the latter is a simple $U\text{-}U$-bimodule by \cite[Theorem~4.15(iv)]{Ma}. 
Applying \cite[Theorem~7]{LZ} it thus follows that 
$\mathcal{L}(M,M)\otimes \mathcal{L}(\mathbf{M},\mathbf{M})$ is a simple $U\text{-}U$-bimodule. Note that
\begin{displaymath}
\mathcal{L}(M,M)\otimes \mathcal{L}(\mathbf{M},\mathbf{M})\subset \mathcal{L}(M\otimes \mathbf{M},M\otimes \mathbf{M}).
\end{displaymath}
The module $\mathcal{L}(M,M)^{\mathrm{ad}}$ is the multiplicity-free
direct sum of all simple finite dimensional $\mathfrak{sl}_2$-modules of odd dimension (this follows
from \eqref{eq505}). From Lemma~\ref{lem753} we thus get
\begin{displaymath}
[\mathcal{L}(M,M)\otimes \mathcal{L}(\mathbf{M},\mathbf{M})]^{\mathrm{ad}}\cong
I^f(0)\oplus I^f(2)\oplus I^f(4)\oplus\dots.
\end{displaymath}
Comparing with the proof of Proposition~\ref{prop327}, we get 
\begin{displaymath}
\mathcal{L}(M,M)\otimes \mathcal{L}(\mathbf{M},\mathbf{M})\cong \mathcal{L}(M\otimes \mathbf{M},M\otimes \mathbf{M}). 
\end{displaymath}
In particular, $\mathcal{L}(M\otimes \mathbf{M},M\otimes \mathbf{M})$ is a simple $U\text{-}U$-bimodule.
As $U_{\lambda}$ is a $U\text{-}U$-subbimodule of $\mathcal{L}(M\otimes \mathbf{M},M\otimes \mathbf{M})$ 
by Theorem~\ref{thm401},
claim~\eqref{thm777.1} follows for  $\lambda(h)\not\in\frac{1}{2}+\mathbb{Z}$. Similar arguments
apply in the case $\lambda(h)=-\frac{3}{2}$.

If $M$ is a Verma $\mathfrak{sl}_2$-module with integral non-negative highest weight, say $k$, then the
$U\text{-}U$-bimodule $\mathcal{L}(M,M)$ has length two by \cite[Theorem~4.15(v)]{Ma}. As tensoring with
$\mathcal{L}(\mathbf{M},\mathbf{M})$ over a field is exact, from \cite[Theorem~7]{LZ} we get that the $U\text{-}U$-bimodule 
$\mathcal{L}(M,M)\otimes \mathcal{L}(\mathbf{M},\mathbf{M})$ has length two. Similarly to the previous paragraph one
shows that 
\begin{displaymath}
\mathcal{L}(M,M)\otimes \mathcal{L}(\mathbf{M},\mathbf{M})\cong \mathcal{L}(M\otimes \mathbf{M},M\otimes \mathbf{M}). 
\end{displaymath}
Hence $U_{\lambda}$ has one proper ideal, call it $J$. Let $\nu$ be such that 
$\nu(z)=\dot{z}$, $\nu(h)=\mu(h)+k$. From the above, $L(\nu)$ is the tensor product of $\mathbf{M}$ with the
$k+1$-dimensional simple $\mathfrak{sl}_2$-module and hence
\begin{displaymath}
\mathcal{L}(L(\nu),L(\nu)) ^{\mathrm{ad}}\cong I^f(k),
\end{displaymath}
in particular, the annihilator of $L(\nu)$ must be different from (in fact, strictly bigger than) 
the annihilator of $\Delta(\nu)$. Therefore $J=\mathrm{Ann}_{U_{\lambda}}L(\nu)$ is primitive.
This completes the proof.
\end{proof}

As an immediate consequence we get:

\begin{corollary}\label{cor764}
Primitive ideals in $U$ with nonzero central charge are exactly the annihilators of simple
highest weight modules with nonzero central charge.
\end{corollary}

\subsection{On primitive ideals for zero central charge}\label{s6.5}

We expect that the problem of classification of primitive ideals in $U$ for zero central charge might
be very difficult. We note that Corollary~\ref{cor764} does not hold for zero central charge. 
Indeed, simple highest weight modules for zero central charge are exactly the simple $\mathfrak{sl}_2$-modules
and they all are annihilated by $\mathfrak{i}$. In \cite[Section~4]{LMZ} one finds many simple weight
$U$-modules with zero central charge whose annihilators do not contain $\mathfrak{i}$.

\vspace{5mm}

\noindent
B.D.: Department of Math., Uppsala University,
Box 480, SE-751 06, Uppsala, Sweden; e-mail: {\tt brendan.frisk.dubsky\symbol{64}math.uu.se}
\vspace{2mm}

\noindent
R.L.: Department of Math., Soochow university, Suzhou 215006,
Jiangsu, P. R. China; e-mail: {\tt rencail\symbol{64}amss.ac.cn}
\vspace{2mm}

\noindent
V.M.: Department of Math., Uppsala University,
Box 480, SE-751 06, Uppsala, Sweden; e-mail: {\tt mazor\symbol{64}math.uu.se}
\vspace{2mm}

\noindent K.Z.: Department of Math., Wilfrid Laurier
University, Waterloo, Ontario, N2L 3C5, Canada; and College of Math. and
Information Science, Hebei Normal (Teachers) University, Shijiazhuang 050016,
Hebei, P. R. China. e-mail:  {\tt kzhao\symbol{64}wlu.ca}

\end{document}